\documentclass[11pt, a4paper]{article}
\usepackage[english]{babel}
\usepackage[T1]{fontenc}
\usepackage[utf8]{inputenc}

\usepackage{textgreek}
\usepackage{sidecap}
\usepackage{titlesec}
\usepackage[usenames,dvipsnames,svgnames,table]{xcolor}
\usepackage{graphicx}
\usepackage{caption}
\usepackage{subcaption}
\usepackage{siunitx}
\usepackage{xcolor}
\usepackage{enumitem}

\usepackage{graphicx}
\usepackage{rotating}

\usepackage{caption}
\usepackage{amsmath}
\usepackage{amsfonts}
\usepackage{amsthm}
\usepackage{mathtools}
\usepackage{amssymb}
\usepackage{mathrsfs}
\usepackage{dsfont}

\usepackage{algorithm,algorithmic}

\allowdisplaybreaks

\newtheoremstyle{mythm}%
{3pt}
{3pt}
{\itshape\color{black}}
{}
{\bfseries\color{blue}}
{.}
{.5em}
{}

\newtheoremstyle{mydef}%
{3pt}
{3pt}
{\upshape\color{black}}
{}
{\bfseries\color{blue}}
{.}
{.5em}
{}

\theoremstyle{mythm}

\newtheorem{theorem}{Theorem}[section] 
\newtheorem{corollary}[theorem]{Corollary}
\newtheorem{proposition}[theorem]{Proposition}
\newtheorem{lemma}[theorem]{Lemma}
\newtheorem*{theorem*}{Theorem}
\newtheorem*{proposition*}{Proposition}

\theoremstyle{mydef}
\newtheorem{definition}[theorem]{Definition}

\numberwithin{equation}{section}

\usepackage{pgf,tikz}
\usetikzlibrary{patterns}
\usepackage{tikz-qtree}
\usetikzlibrary{calc,patterns,angles,quotes}

\usepackage{float}
\usetikzlibrary{shapes.misc}
\tikzstyle{root}=[circle,fill=black,inner sep=0pt,minimum size=8pt]
\tikzstyle{steiner}=[circle,fill=blue,inner sep=0pt,minimum size=8pt]
\tikzstyle{terminal}=[fill=red]

\tikzstyle{X}=[circle,fill=blue,inner sep=0pt,minimum size=6pt]
\tikzstyle{T}=[fill=red,inner sep=0pt,minimum size=4pt]
\tikzstyle{C}=[fill=red,inner sep=0pt,minimum size=4pt]
\tikzstyle{R}=[circle,fill=black,inner sep=0pt,minimum size=4pt]
\tikzstyle{background-line}=[color=green!55]

\tikzset{cross/.style={cross out, draw=black, minimum size=2*(#1-\pgflinewidth), inner sep=0pt, outer sep=0pt},
	cross/.default={2pt}}
	
\usepackage[dvipsnames]{xcolor}

\definecolor{mygray}{gray}{0.28}
\definecolor{myorange}{RGB}{220,120,50}

\newcommand{\oc}{i_c}
\newcommand{\ta}{z}
\newcommand{\ipinf}{i_{+\infty}}
\newcommand{\iminf}{i_{-\infty}}
\newcommand{\apr}{b}

\usepackage{nccfoots}
\usepackage{scalerel}
\usetikzlibrary{svg.path}

\definecolor{orcidlogocol}{HTML}{A6CE39}
\tikzset{
  orcidlogo/.pic={
    \fill[orcidlogocol] svg{M256,128c0,70.7-57.3,128-128,128C57.3,256,0,198.7,0,128C0,57.3,57.3,0,128,0C198.7,0,256,57.3,256,128z};
    \fill[white] svg{M86.3,186.2H70.9V79.1h15.4v48.4V186.2z}
                 svg{M108.9,79.1h41.6c39.6,0,57,28.3,57,53.6c0,27.5-21.5,53.6-56.8,53.6h-41.8V79.1z M124.3,172.4h24.5c34.9,0,42.9-26.5,42.9-39.7c0-21.5-13.7-39.7-43.7-39.7h-23.7V172.4z}
                 svg{M88.7,56.8c0,5.5-4.5,10.1-10.1,10.1c-5.6,0-10.1-4.6-10.1-10.1c0-5.6,4.5-10.1,10.1-10.1C84.2,46.7,88.7,51.3,88.7,56.8z};
  }
}

\newcommand\orcidicon[1]{\href{https://orcid.org/#1}{\mbox{\scalerel*{
\begin{tikzpicture}[yscale=-1,transform shape]
\pic{orcidlogo};
\end{tikzpicture}
}{|}}}}

\usepackage{hyperref}

\title{\vspace{-2.2cm}\textbf{Traveling waves of an FKPP-type model\\for self-organized growth}}
\author{Florian Kreten \thanks{Institut für Angewandte Mathematik, Rheinische Friedrich-Wilhelms-Universität, Endenicher Allee 60,
53115 Bonn, Germany. Email: \href{mailto:florian.kreten@uni-bonn.de}{florian.kreten@uni-bonn.de}. \newline
\indent This work was partly funded by the Deutsche Forschungsgemeinschaft (DFG, German Research Foundation) under Germany’s Excellence Strategy - GZ 2047/1, Projekt-ID 390685813 and by the Deutsche Forschungsgemeinschaft (DFG, German Research \mbox{Foundation}) - Projektnummer 211504053 - SFB 1060.} \, \orcidicon{0000-0003-1938-2590}}

\begin{document}
\renewcommand{\abstractname}{\vspace{-\baselineskip}}
\maketitle

\begin{abstract}
\vspace{-0.5cm}
\noindent \textbf{Abstract}: We consider a reaction-diffusion system of densities of two types of particles, introduced by Edouard Hannezo et al. in the context of branching morphogenesis (Cell, 171(1):242–255.e27, 2017). It is a simple model for a growth process: active, branching particles form the growing boundary layer of an otherwise static tissue, represented by inactive particles. The active particles diffuse, branch and become irreversibly inactive upon collision with a particle of arbitrary type. In absence of active particles, this system is in a steady state, without any a priori restriction on the amount of remaining inactive particles. Thus, while related to the well-studied FKPP-equation, this system features a game-changing continuum of steady state solutions, where each corresponds to a possible outcome of the growth process. However, simulations indicate that this system self-organizes: traveling fronts with fixed shape arise under a wide range of initial data. In the present work, we describe all positive and bounded traveling wave solutions, and obtain necessary and sufficient conditions for their existence. We find a surprisingly simple symmetry in the pairs of steady states which are joined via heteroclinic wave orbits. Our approach is constructive: we first prove the existence of almost constant solutions and then extend our results via a continuity argument along the continuum of limiting points.\\
	
\noindent \textbf{Key words}: Developmental biology, pattern formation, cellular organization, traveling wave, reaction-diffusion equation, continuum of fixed points.\\
\noindent \textbf{MSC2020}: 92C15, 35C07, 35K57, 34C14.
\end{abstract}


\section{Motivation and result}

The mechanics of tissue-growth have drawn the attention of the scientific community. A central question is, how the cells are organized, how they react to and communicate with their environment on the microscopic level, and how their behavior during the growth phase gives rise to distinct macroscopic structures. Mathematical models can help to understand these processes. Works regarding organoids, wound healing or tumor growth are abundant \cite{Olivas_Organoid_MathModels, Mammoto2010_Organ_Development, Falco2021_Clinical_Cancer_Modeling_Glioblastoma, DAlessandro2021_Cell_Memory_Migration, Ladoux_Cell_Migration_Plasticity}. However, for most of these models, our numerical skills far predominate the possibility to analyze them rigorously. Hence, for understanding the basic mechanics of the underlying biological processes, the need for simplified models arises.

Especially when studying spatiotemporal effects and macroscopic pattern formation, reaction-diffusion systems and their traveling waves have proven insightful. One of the oldest and most studied models is the FKPP-equation \cite{Fisher_1937_Wave, KPP_1937_Wave}, describing the advance of an advantageous population. The arise of more complex spatial patterns due to the instability of a homogeneous state was first described in Turings groundbreaking paper \textit{The Chemical basis of Morphogenesis} \cite{Turing_pattern_formation_morghogenesis}. More recently, systems of Keller-Segel type have been studied extensively, where growth, movement and self-organization of a population are driven by chemotactic guidance \cite{Keller_bacteria, Perthame_Keller_Segel_2004, Painter_Keller_Segel_SelfOrganization_2019}. We want to highlight the works of Painter \cite{Painter_Keller_Segel_SelfOrganization_2019} and of Othmer \textit{et al.} \cite{Othmer2009WavesinBiology} for an impression of mathematical modeling of pattern formation in developmental biology.

The group of Hannezo \textit{et al.} proposed \textit{A Unifying Theory of Branching Morphogenesis} in epithelial tissues \cite{Hannezo_2017_Unifying}. They introduced a stochastic model, related to branching and annihilating random walks \cite{Cardy_Tauber_BARW}. In this model, a branched structure is represented by a network. This network undergoes stochastic growth dynamics, where each branch of the network grows independently from the others and follows a set of simple, local rules. At its tip, each branch elongates or splits up at certain rates and these tips are called \textit{active}. When an active tip comes too close to a different branch, it irreversibly ceases any activity and becomes \textit{inactive}. The numerical results of Hannezo \textit{et al.} reveal that this stochastic growth process self-organizes: the active tips are concentrated at the boundary of the network and form a rather sharp layer of growth. The center of the network is static and - rather surprisingly - exhibits a homogenous geometry, in particular a constant density of branches. Remarkably, as mentioned by the authors, this model self-organizes without any signaling gradients. Even a directional bias of the branches can be achieved, as the result of an appropriate spatial boundary. Moreover, the authors observed that their simulations were in good agreement with biological data from mammary glands, kidneys and the human prostate \cite{Hannezo_2017_Unifying}.

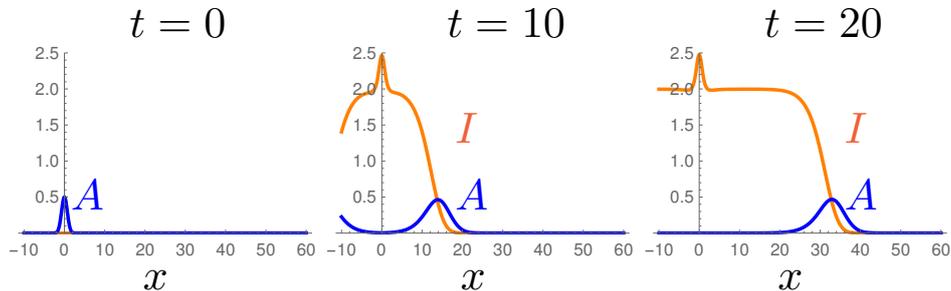
\begin{figure}[h]
 		\centering
 		\begin{minipage}[c]{0.32\textwidth}
\begin{picture}(100,100)
	\put(0,0){\includegraphics[width=\textwidth]{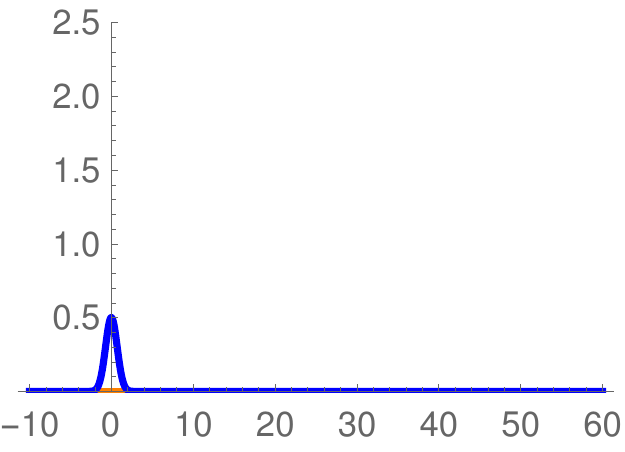}}
	\put(50,-10){\scalebox{1.5}{$x$}}
	\put(45,85){\scalebox{1.5}{$t = 0$}}
	\put(23,20){{\scalebox{1.5}{${\color{blue}A}$}}}
\end{picture} 	
 		\end{minipage}
 		\begin{minipage}[c]{0.32\textwidth}
\begin{picture}(100,100)
\put(0,0){\includegraphics[width=\textwidth]{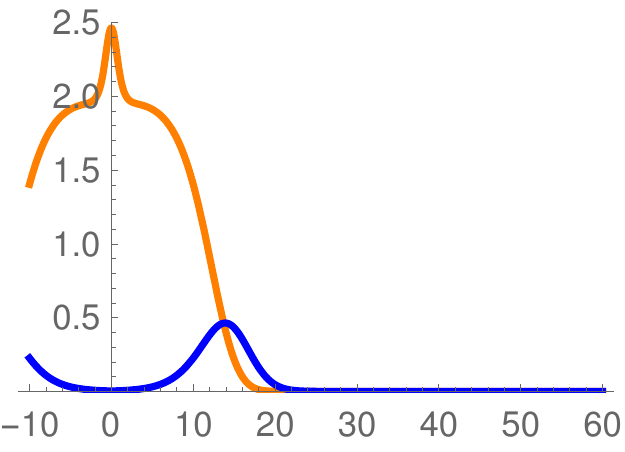}}
	\put(50,-10){\scalebox{1.5}{$x$}}
	\put(45,85){\scalebox{1.5}{$t = 10$}}
	\put(48, 45){{\scalebox{1.5}{${\color{RedOrange}I}$}}}
	\put(48, 20){{\scalebox{1.5}{${\color{blue}A}$}}}
\end{picture} 	
 		\end{minipage}
 		\begin{minipage}[c]{0.32\textwidth}
\begin{picture}(100,100)
 		\put(0,0){\includegraphics[width=\textwidth]{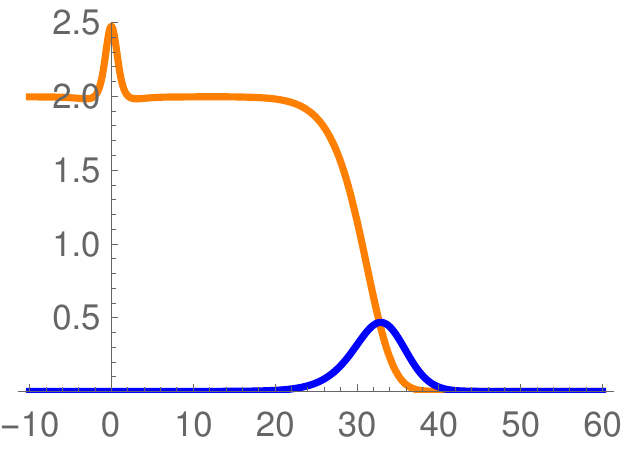}}
	\put(50,-10){\scalebox{1.5}{$x$}}
	\put(45,85){\scalebox{1.5}{$t = 20$}}
	\put(75, 45){{\scalebox{1.5}{${\color{RedOrange}I}$}}}
	\put(75, 20){{\scalebox{1.5}{${\color{blue}A}$}}}
\end{picture} 	
 		\end{minipage}
 	\vspace{0.2cm}
	\caption{Simulation of the Reaction-Diffusion System \eqref{EQUA} for $r = 0$. Given a small initial heap of active particles $A(x,0)=1/2 \exp (-x^2)$ and $I(x,0) = 0$, two identical traveling fronts arise, the right one is shown. After the separation of the two fronts away from the origin, the density of the remaining inactive particles is given by $I = 2$ and the front moves asymptotically with speed $c=2$.}
 		\label{Diff_system_pics}
\end{figure}

To study their model analytically, Hannezo \textit{et al.} proposed the following system, which corresponds to the diffusive limit of the above stochastic dynamics. We restrict ourselves to the one-dimensional case. Due to a simple linear rescaling (Appendix \ref{ch:Rescaling}), we only need to consider the normalized reaction-diffusion system
\begin{equation}
\begin{aligned}
		A_t &= A_{xx}  + A - A(A+I), \\
		I_t &= A(A+I) + rA. \label{EQUA}
\end{aligned}
\end{equation}
Here, $A,I: \mathds{R} \times \mathds{R}^+ \rightarrow \mathds{R}^+$ are the densities of \textit{active particles} and \textit{inactive particles}. The diffusion term describes the movement of the active particles, all other terms encode a growth process where the active particles eventually become inactive: the active particles they branch with rate $1$, produce inactive particles with rate $r \geq 0$, and become inactive upon collision with either an active or an inactive particle. The active particles grow logistically, which implies that the inactive particles grow at most exponentially. The resulting simple time-dependent bounds on $A$ and $I$ yield uniqueness and existence of smooth solutions, a suitable fixed-point theorem is presented in chapter 14 of \cite{Smoller_Shocks_Reaction_Diffusion}. More details about the underlying stochastic processes together with a non-rigorous derivation of this PDE can be found in \cite{Hannezo_2017_Unifying}. Note that without the inactive particles, \textit{i.e.} when $I = 0$, the remaining equation for $A$ reduces to the well-known FKPP-equation  \cite{Fisher_1937_Wave, KPP_1937_Wave}.

The System \eqref{EQUA} can be interpreted as a twofold degenerate Keller-Segel system \cite{Keller_bacteria, Arumugam_2020_Keller-Segel}: the active particles are not guided by a chemotactic gradient, but explore the space solely diffusively, and the inactive particles do not diffuse at all. Still, simulations of System \eqref{EQUA} show that general solutions of \eqref{EQUA} self-organize, which is typical for many different Keller-Segel systems \cite{Painter_Keller_Segel_SelfOrganization_2019}. The invading front of the system converges to a fixed shape: a pulse of active particles, that represents a layer of growth, is accompanied by a monotone wave of inactive particles, the resulting static tissue, as demonstrated in Figure \ref{Diff_system_pics}. In this sense, the Reaction-Diffusion Equation qualitatively reproduces the macroscopic behavior of the stochastic dynamics, which models the well-behaved growth of an epithelial tissue.

For a wave speed $c>0$, a right-traveling wave solves Eq. \eqref{EQUA} via the Ansatz $A(x,t) = a(x-ct), I(x,t) = i(x-ct)$. We substitute $z=x - ct $, such that any traveling wave must be a solution of
\begin{equation}
\begin{aligned}
	&0 = a_{zz}  + c a_z +a -a  (a+i), \\
	&0 = c i_z + a  (a+i) + r  a.
\end{aligned} \label{WAVE_EQ}
\end{equation}

The occurrence of these seemingly stable traveling waves is quite surprising, since the System \eqref{EQUA} features a continuum of steady state solutions:
\begin{align}
A=0, \, I=K, \,  K \in \mathds{R}^+, \label{intro_steady_state_continuum}
\end{align}
which is due to the fact that the inactive particles do not degrade. This continuum of steady states represents the difficulty when studying the system: we first need to find out which limiting states are chosen by the growth process.

Hannezo \textit{et al.} presented a rich discussion of the Wave-Equation \eqref{WAVE_EQ} along with numerics and several heuristics that show a deep connection with the original FKPP-equation, and predicted some of the following results. The goal of this paper is to give necessary and sufficient conditions for the existence of such traveling wave solutions and to analyze the shape of the wave form. Our main result characterizes a family of pulled traveling waves:

\begin{theorem} \label{Main_Theorem}
Let $r \geq 0, c > 0$ and consider the System \eqref{EQUA} and its traveling wave solutions given by \eqref{WAVE_EQ}. Set $\oc : = \max \{ 0, 1-c^2/4 \} $. For each pair $\iminf, \ipinf \in \mathds{R}^+$ such that
\begin{align}
\ipinf \in [\oc,1), \qquad \iminf = 2 - \ipinf, \label{main_thm_relation}
\end{align}
there exists a unique bounded and positive traveling wave $a,i \in C^\infty(\mathds{R}, \mathds{R}^2)$ with speed $c$ such that
\begin{align}
\lim_{\ta \rightarrow \pm \infty} & a(z) = 0, \qquad  \lim_{\ta \rightarrow \pm \infty} i(z) = i_{\pm \infty}. \label{intro_main_them_limits}
\end{align}
The function $i(z)$ is decreasing, whereas $a(z)$ has a unique local and global maximum. If $\frac{c^2}{4} + \ipinf -1  = 0$, then convergence as $z \rightarrow + \infty$ is sub-exponentially fast and of order $ z \cdot e^{ - \frac{c}{2} z}$. If $ \frac{c^2}{4} + \ipinf -1  > 0$, then convergence as $z \rightarrow + \infty$ is exponentially fast. Convergence as $z \rightarrow - \infty$ is exponentially fast in all cases. The corresponding rates are
\begin{align}
\mu_{\pm \infty} = - \frac{c}{2} + \sqrt{ \frac{c^2}{4} + i_{\pm \infty} -1}.
\end{align}
Moreover, these are all bounded, non-negative, non-constant and twice differentiable solutions of Eq. \eqref{WAVE_EQ}.
\end{theorem}

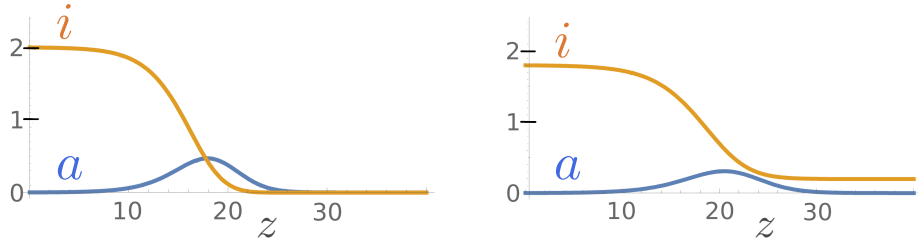
\begin{figure}[h]
 		\centering
  		\begin{minipage}[c]{0.45\textwidth}
\begin{picture}(100,100)
	\put(0,0){\includegraphics[width=\textwidth]{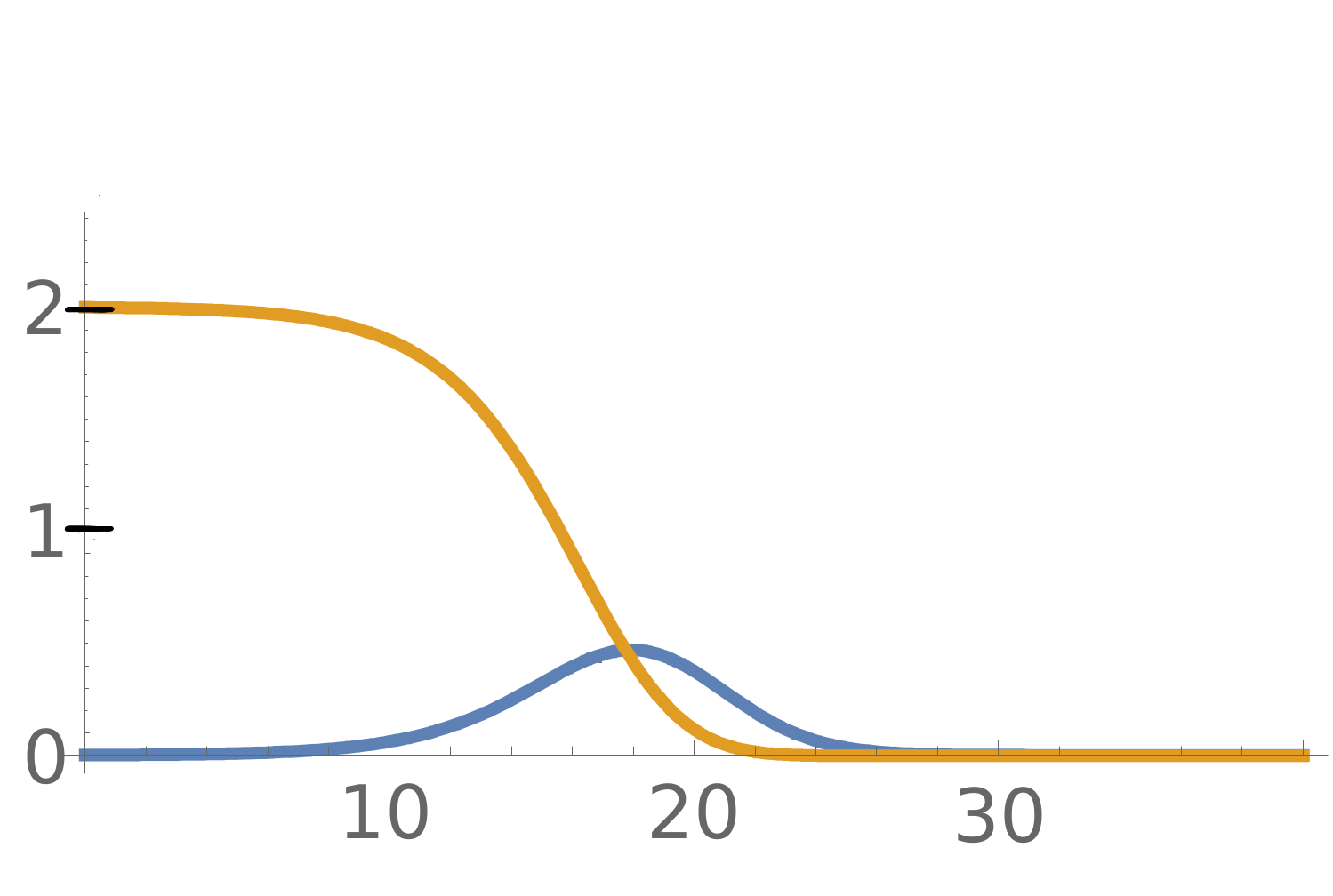}}
	\put(95,0){{\scalebox{1.5}{${\color{mygray}\ta}$}}}
	\put(20, 75){{\scalebox{1.8}{${\color{myorange}i}$}}}
	\put(20, 22){{\scalebox{1.8}{${\color{RoyalBlue}a}$}}}
\end{picture} 	  		
 		\end{minipage}
 		\hspace{0.5cm}
  		\begin{minipage}[c]{0.45\textwidth}
\begin{picture}(100,100)
	\put(0,-6){ \includegraphics[width=0.98\textwidth]{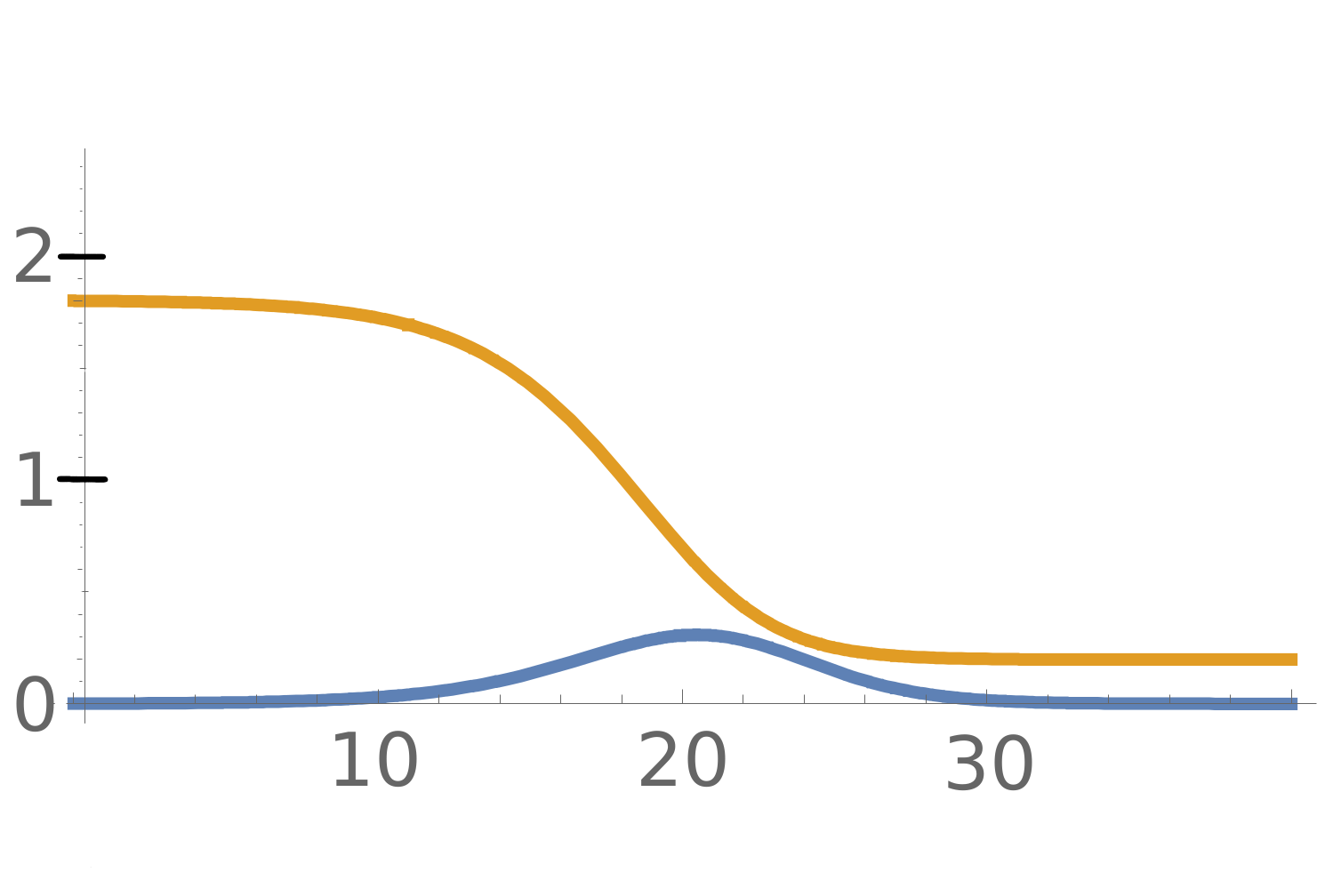}}
	\put(98,0){{\scalebox{1.5}{${\color{mygray}\ta}$}}}
	\put(23, 68){{\scalebox{1.8}{${\color{myorange}i}$}}}
	\put(23, 22){{\scalebox{1.8}{${\color{RoyalBlue}a}$}}}
\end{picture}
\end{minipage}
 		\caption{Two different traveling waves with speed $c = 2$. The limits of the left wave are given by $\iminf = 2$ and $\ipinf = 0$. The limits of the right wave are given by $\iminf = 1.8$ and $\ipinf = 0.2$.} 
 		\label{fig:trav_waves}
\end{figure}

Notice that this result is independent of the reproduction rate $r$, which affects the shape of the wave, but neither its limits nor the minimal speed of a positive solution. Hence, all non-negative and bounded traveling waves resemble the ones depicted in Figure \ref{fig:trav_waves}, consisting of a pulse of active particles and a monotone wave of inactive particles. These traveling wave solutions share many similarities with classical FKPP-waves of a single type of particles. Among other mathematical aspects, this will be discussed at the end of the paper, in Section \ref{ch:discussion}. Notably, Theorem \ref{Main_Theorem} analytically connects two continua of fixed points via a continuum of traveling waves. Our constructive approach is a novelty: we first prove the existence of almost constant solutions and then continuously deform these solution along the continuum of possible limits.

Figure \ref{Diff_system_pics} shows a simulation of the System \eqref{EQUA}, starting with a small initial amount of active particles. After a short transition phase, we observe a front with fixed shape. Asymptotically, it equals the unique traveling wave with limits $\iminf = 2, \ipinf = 0$ and speed $c=2$, which is the minimal possible wave speed for this pair of limits. We observed this behavior for all compact initial data that we chose. Moreover, this wave seems to be stable against perturbations, as briefly discussed in the concluding Section \ref{ch:discussion}. Even though it is only a first step into this direction, this paper sheds light at the ability of the Growth-Process \eqref{EQUA} to self-organize and at the robustness of this mechanism, e.g. against errors of individual particles. Our theoretical analysis fortifies the numerical and biological findings of Hannezo \textit{et al.}, where a simple set of local rules organizes the growth of a complex epithelial structure. The underlying assumption of a logistic growth is quite natural, so similar rules might drive and regulate other growth processes as well, without the need for guiding gradients.

\section{Outline of the paper}
\label{Sec:_Outline_full}

A sketch of the central ideas and techniques is presented in Section \ref{Sec:_Outline_full}. The identity $\iminf + \ipinf = 2$ is proved in Section \ref{ch:nec_cond_easy}. The asymptotic behavior around the stable and unstable set of the traveling waves is analyzed in Section \ref{ch:fix_point_results}. A non-negative trapping region of a lower-dimensional system in coordinates $(a, b)$ is analyzed in Section \ref{ch:Attractor}. We use our knowledge about the lower-dimensional system to prove the existence of a suitable attractor of the full system in Section \ref{ch:full_system_attractor}. Then, we connect the unstable manifold of the unstable set with this attractor, see Section \ref{ch:Connection}. We complete the proof of Theorem \ref{Main_Theorem} in Section \ref{ch:main_thm_proof}. In Section \ref{ch:discussion}, we highlight the similarity of the traveling waves with those of the original FKPP-equation and give a short outlook at their stability.

\subsection{Identifying the correct limits} \label{Sec:_Outline_1}

We reformulate the System \eqref{WAVE_EQ} for a traveling wave as an equivalent system of first-order ODEs. Denoting differention with respect to $z$ by a prime, we introduce the auxiliary variable $a'=\apr$, so that 
\eqref{WAVE_EQ} becomes
\begin{align}
a' &= b, \nonumber \\
b' &= a(a+i) - a - cb, \label{3D_Flow_Def} \\
i' &= -\frac{1}{c} a \left (a+i+r \right ), \nonumber
\end{align}
for some $c>0,r \geq 0$. We call a solution of Eq. \eqref{3D_Flow_Def} non-negative if $a,i \geq 0$. If a solution is in $C^1(\mathds{R}, \mathds{R}^3)$, then it is also in $C^\infty$ by a simple induction. This equation has a continuum of non-negative fixed points, similar to that of the PDE, cf. \eqref{intro_steady_state_continuum}:
\begin{align}
a=b=0, \quad i \in \mathds{R^+}. \label{intro_fixed_points_ODE}
\end{align}

Thus, in the first place, we need to find out which of these fixed points can be considered as limits of right-traveling waves. Any bounded and non-negative solution of System \eqref{3D_Flow_Def} can not be periodic and must converge since $ci'=-  a(a+i+r) \leq 0$. It is now evident that the limits at $z = \pm \infty$ must be fixed points of Eq. \eqref{3D_Flow_Def}, thus we denote them as $(a,\apr,i)=(0,0,i_{\pm \infty})$. Under mild assumptions regarding integrability, we can interrelate two different points on a given traveling wave, see Section \ref{ch:nec_cond_easy}. Most importantly, this leads to the correspondence of the limits
\begin{align}
\ipinf + \iminf=2. \label{nec_cond_intro_sum}
\end{align}
In view of this, monotonicity of $i$ implies that $\iminf \in (1,2]$ and $\ipinf \in [0,1)$.

The fixed points of the ODE System \eqref{3D_Flow_Def} are not isolated, hence its Jacobian $D$ is degenerate there. It is easily verified that $D$ is given by
\begin{align}
D_{(a,\apr,i)} & =
\begin{pmatrix}
0 & 1 & 0 \\
2a + i -1 & -c & a \\
-\frac{1}{c}(2a+i+r) & 0 & - \frac{a}{c}
\end{pmatrix}.
\end{align}
At a fixed point $(a,b,i) = (0,0,K)$, the eigenvalues of $D_{(0,0,K)}$ are
\begin{align}
\lambda_0 = 0, \hspace{1cm}  \lambda_\pm = -\frac{c}{2} \pm \sqrt{\frac{c^2}{4} + K -1}. \label{Eigenvalues_DF_intro}
\end{align}
Hence, we can not apply the classical Theorem of Grobmann-Hartmann to linearize the asymptotic behavior around the fixed points. We apply center manifold theory to work out the higher moments of the approximation, see Section \ref{ch:fix_point_results}. The center manifold coincides with the continuum of fixed points. This implies that asymptotically, there is no flow along the direction of the eigenvector $(a,\apr,i) = (0,0,1)$ with zero eigenvalue. Hence, the asymptotic flow around any fixed point is two-dimensional and the stability of the fixed point $(0,0,K)$ is dictated by the two eigenvalues $\lambda_\pm$. When $K>1$, the fixed point is unstable, while for $K<1$, it is stable.

At the same time, the analysis of the asymptotic behavior also yields a necessary condition on the speed $c$ of a non-negative wave. A traveling wave can only be non-negative if $a(\ta)$ does not spiral while converging to $0$. Therefore, the two eigenvalues $\lambda_\pm$ at the limiting fixed point must be real-valued. In view of \eqref{Eigenvalues_DF_intro}, for a fixed point $(0,0,K)$, this is given if
\begin{align}
\frac{c^2}{4} + K -1 \geq 0. \label{nec_cond_intro_oscillate}
\end{align}
Thus, if the stable fixed point $(0, 0, \ipinf)$ is the limit of a non-negative traveling wave, where $\ipinf \in [0,1)$, it must by \eqref{nec_cond_intro_oscillate} further hold that
\begin{align}
\frac{c^2}{4} +\ipinf  -1 \geq  0 \quad  \Leftrightarrow \quad \ipinf \geq \oc=\max \{ 0, 1-\frac{c^2}{4} \},
\end{align}
as in Theorem \ref{Main_Theorem}. In other words, $\oc$ is the minimal limiting density of inactive particles that is necessary for the existence of a non-negative traveling wave with speed $c$.

\subsection{Construction of a traveling wave} 
\label{Sec:_Outline_2}
 
We will explicitly construct a non-negative traveling wave such that the two necessary conditions $\ipinf \geq \oc$ and $\ipinf + \iminf = 2$ are fulfilled. Two key features of the model make it tractable: first, the monotonicity of $i(z)$ allows us to investigate convergence of a sub-system that arises for a fixed value of $i$, and then lift our result to almost constant solutions of the full system. Second, for extending this result to non-small solutions, we lean on an integral equation that allows us to interrelate two points on a given trajectory. However, the central Proposition \ref{Prop_integrals} depends essentially on the logistic growth of the active particles. Apart from this, our approach seems to be applicable to a broader class of systems.

Regarding the ODE System \eqref{3D_Flow_Def}, our analysis of the flow around the fixed points in Section \ref{ch:fix_point_results} reveals a suitable \textit{unstable set}
\begin{align}
S_{-\infty} &:= \big\{ (0,0,i) : \, i \in (1,2] \big \}, \label{eq_unstable_set}\\
\intertext{and a suitable \textit{stable set}}
S_{+\infty} &:= \big \{ (0,0,i) : \, i \in [0,1) \big \}.
\label{eq_stable_set}
\end{align}
Each point $(0,0,\iminf) \in S_{-\infty}$ has an unstable manifold of dimension one. Its restriction to $ a \geq 0 $ is the only possible candidate for the tail of a non-negative traveling wave as $\ta \rightarrow - \infty$. Each point $(0,0,\ipinf) \in S_{+\infty}$ is Lyapunov stable, which can also be seen in Figure \ref{fig:intro_wave_phase_plot}.

\begin{figure}[h]
\vspace{2cm}
 		\centering
 		\begin{minipage}[c]{0.45\textwidth}
  \begin{picture}(100,100)
	\put(0,0){\includegraphics[width=0.9\textwidth]{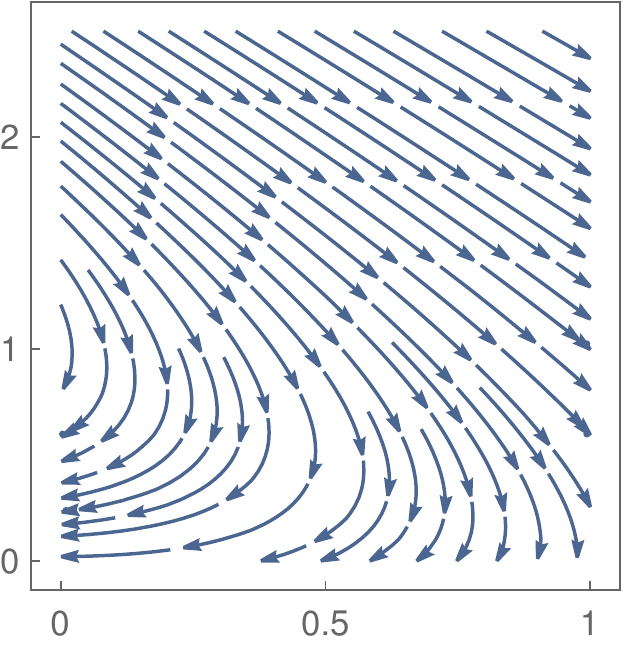}}
	\put(-8,90){\scalebox{1.8}{${\color{mygray}i}$}}
	\put(90,-5){\scalebox{1.8}{${\color{mygray}a}$}}
\end{picture}
\end{minipage}
\vspace{0.3cm}
 		\caption{Two-dimensional phase portrait of $(a,i)$ of traveling waves \eqref{3D_Flow_Def} for $c=2$ and $r=0$, omitting the coordinate $b=a'$. A unique trajectory emerges from each point in $S_{- \infty}$ (where $\iminf >1$) in positive direction of $a$ and converges to $S_{+ \infty}$ (where $\ipinf <1$). Notice the correspondence $\iminf + \ipinf = 2$ of the limits.}
 		\label{fig:intro_wave_phase_plot}
\end{figure}

To begin with, we let $(0,0,\iminf) \in S_{ - \infty}$, where $\iminf \in (1,2-\oc]$, and follow its unstable manifold in positive direction of $a$. We prove that there is a finite phase-time $\ta_0$ such that $\apr(\ta_0) = 0$: the trajectory reaches a local maximum of active particles, again see Figure \ref{fig:intro_wave_phase_plot}. We denote it as $(a_{\ta_0}, 0,i_{\ta_0})$. This is carried out in Section \ref{ch:Connection}.

Thus, for finding a suitable attractor of $S_{+\infty}$, we analyze solutions that start in points of type $(a,b,i) = (a_0,0,i_0)$. We first analyze the lower-dimensional subsystem in coordinates $(a,b)$, imposing a fixed value of $i$, which is done Section \ref{ch:Attractor}. The monotonicity of $i(z)$ allows us to lift this result to the full system, see Section \ref{ch:full_system_attractor}. Here, the continuum of fixed points comes at help: we first prove the existence of almost constant solutions, where $a \ll 1$ and $i \sim i_0$, that stay non-negative and converge. Then, we continuously deform these solutions: the Lyapunov-stability of the limits in $S_{+\infty}$ implies continuity of the entire trajectory up to $z =  + \infty$ in initial data. This results in sharp conditions regarding $(a_0,0,i_0)$ such that the trajectory stays non-negative and converges.

We show that the first local maximum $(a_{\ta_0}, 0,i_{\ta_0})$ along the instable manifold of $(0,0,\iminf)$ does fulfill these conditions, see again Section \ref{ch:Connection}. The technique is the same as for proving the identity $\ipinf + \iminf = 2$, which is the starting point of our analysis and presented in the next section. Finally, the proof of Theorem \ref{Main_Theorem} is completed in Section \ref{ch:main_thm_proof}, when we bring together all the different pieces. The resulting continuous family of solutions is sketched in Figure \ref{fig:intro_wave_phase_plot}.

\section{The mapping of the limits $\boldsymbol{\iminf + \ipinf = 2}$} \label{ch:nec_cond_easy}

We first verify global integrability of a non-negative solution:

\begin{lemma} \label{Lem_global_integrability}
Let $a(\ta),b(\ta),i(\ta)$ be a smooth, bounded and non-negative traveling wave that solves the ODE System \eqref{3D_Flow_Def}. Then, as $\ta \rightarrow \pm \infty$, $a(\ta)$ vanishes and $i(\ta)$ converges, and $a,b,b',i' \in L^1(\mathds{R})$. Moreover, $i(\ta)$ is decreasing and $a(\ta)$ has a unique global and local maximum.

\begin{proof}
Let $a(\ta),b(\ta),i(\ta)$ be a smooth, bounded and non-negative solution of Eq. \eqref{3D_Flow_Def}. Since $ci' = -a(a+i+r)\leq 0$, it holds that $i(\ta)$ is decreasing and by boundedness must converge as $\ta \rightarrow \pm \infty$, so $i' \in L^1(\mathds{R})$. The two limits must be fixed points, so they are given by some $(a,b,i) = (0,0,\iminf)$ and $(0,0,\ipinf)$. Equality is only given if $a(\ta) \equiv 0$. If not $a(\ta) \equiv 0$, then there is at least one local maximum of active particles, which we denote as $(a_0,0,i_0)$.

At this point, $a'' = b' = a_0 (a_0 + i_0 -1) \leq 0$, so either $a_0=0$ and the wave is constant, or $a_0 + i_0 \leq 1$. In the second case case, assume that there is also a local minimum of $a(\ta)$, denoted as $(a_m,0,i_m)$. Since $a(\ta)$ vanishes as $\ta \rightarrow \pm \infty$, we may assume that this be the first local minimum after passing through $(a_0,0,i_0)$. As before, this is already a fixed point or $a_m + i_m \geq 1$. Since $i(\ta)$ is decreasing, $a(\ta)$ must have been increasing, a contradiction to the assumption that this is the first local minimum after the maximum $(a_0,0,i_0)$. Thus, there is only one local maximum of active particles, which is also the global one. Further, this implies $b \in L^1(\mathds{R})$. By $ci'= -a(a+i+r) \leq 0$, we know that $a(a+i+r)$ is also in $L^1(\mathds{R})$. We integrate $b'+cb +a = a(a+i)$ over the finite interval $[-M,M]$, then send the boundaries to $\pm \infty$:
\begin{align}
	  \int_{-M}^{M} b'(\ta) + c b(\ta) + a(\ta) \, d \ta 
	 =   \int_{-M}^{M} a(\ta) \cdot \big[ a(\ta) + i(\ta) \big] \, d \ta.
\end{align}
We know that the right-hand is integrable since $i' \in L^1(\mathds{R})$, and that both $a(\pm M)$ and $b(\pm M)$ vanish as $M \rightarrow +\infty$. This implies
\begin{align}
 \int_{\mathds{R}} a(\ta) \text{ } d\ta &= 
\lim_{M \rightarrow + \infty} \Big [ b(M) -b(-M) + c \cdot \big [a(M) - a(-M) \big] + \int_{-M}^{M} a(\ta) \, d \ta \Big ] \nonumber \\ 
	  & = \int_{\mathds{R}} a(\ta) \cdot \big[ a(\ta)+i(\ta) \big] \, d\ta .
\end{align}
Hence also $a \in L^1(\mathds{R})$, since $a \geq 0$. Finally, as a sum of integrable terms, also $b' \in L^1(\mathds{R})$.
\end{proof}
\end{lemma}

The following Proposition \ref{Prop_integrals} will be used several times to interrelate two points $(a_1,0,i_1), (a_2,0,i_2)$ on a traveling wave, where $b_i = 0$. By the previous Lemma, the necessary conditions regarding integrability are always verified for non-negative and bounded solutions.

\begin{proposition} \label{Prop_integrals}
Let $a(\ta),b(\ta),i(\ta)$ be a smooth and bounded solution of the ODE System \eqref{3D_Flow_Def} on some interval $[z_1, z_2]$, where $-\infty \leq z_1 \leq z_2 \leq + \infty$. Assume that $b(z_1) = b(z_2) = 0$. Further, assume that $a,b,b',i'$ are integrable and define $\mathscr{A}(t) := \int_{z_1}^t a(\ta) \, d\ta$. The following three identities hold:
\begin{align}
\int_{z_i}^{z_2} a(z) \big [ a(z) + i(z) ] \, dz &= \mathscr{A}(z_2) + c \cdot \big [ a(z_2) - a(z_1) \big], \label{first_mass} \\
i(z_1) - i(z_2)  &= \frac{1+r}{c} \mathscr{A}(z_2), \label{second_mass}\\
\int_{z_i}^{z_2} a(z) \big [ a(z) + i(z) ] &= \big [ i(z_2) + a(z_2) \big] \cdot \mathscr{A}(z_2), \nonumber \\
 & \hspace{0.5cm} + \frac{1+r}{2c} \mathcal{A}(z_2)^2 + \frac{a(z_1)^2 -a(z_2)^2}{2c}. \label{third_mass}
\end{align}
\begin{proof}
Any solution of the ODE System \eqref{3D_Flow_Def} also fulfills the original Wave Equations \eqref{WAVE_EQ}. We integrate these over $[z_1,z_2]$, substitute $a'=b$ and use that $b( z_i) = 0$. This directly proves \eqref{first_mass} and \eqref{second_mass}. Regarding Eq. \eqref{third_mass}, note that by integration by parts:
\begin{align} 
\begin{aligned}
	&\int_{z_1}^{z_2} a(\ta) \cdot \big [ a(\ta) + i(\ta) \big ] \, d\ta = \Big{|}_{z_1}^{z_2} \mathscr{A} \cdot (a+i) \\
	& \hspace{5cm} - \int_{z_1}^{z_2} \mathscr{A}(\ta) \cdot \big [ b(\ta) + i'(\ta) \big ] \, d\ta \\
	&= \big [ i(z_2) + a(z_2) \big] \cdot \mathscr{A}(z_2)
	+ \int_{z_1}^{z_2} \mathscr{A}(\ta) \cdot \frac{1}{c} \big[ (1+r)a(\ta) + b'(\ta) \big] \, d\ta  \\
	&= \big [ i(z_2) + a(z_2) \big] \cdot \mathscr{A}(z_2)
	+ \frac{1+r}{2c} \Big{|}_{z_1}^{z_2} \mathscr{A}^2
	+ \frac{1}{c} \int_{z_1}^{z_2} \mathscr{A}(\ta)b'(\ta) \, d\ta  \\
	&= \big [ i(z_2) + a(z_2) \big] \cdot \mathscr{A}(z_2) + \frac{1+r}{2c} \mathscr{A}(z_2)^2 + \frac{a(z_1)^2 -a(z_2)^2}{2c}.
	\end{aligned}
\end{align}
\end{proof}
\end{proposition}
\textit{Remark}: Equations \eqref{first_mass} and \eqref{second_mass} encode a mass transfer from the active to the inactive particles and are not specific for the chosen reactions. It is the quadratic Eq. \eqref{third_mass} that relies on a logistic saturation mechanism, we do not see a (direct) way to generalize this result.\\

Given Proposition, the identity $\iminf + \ipinf=2$ is a mere

\begin{corollary}[Limits of traveling waves] \label{nec_cond_thm}
Let $a(\ta), b(\ta),i(\ta)$ be a non-negative and bounded traveling wave that solves the ODE System \eqref{3D_Flow_Def}, and denote its limits as $(a,b,i) = (0,0,i_{\pm \infty})$. Either $\int_{\mathds{R}}a(z) \, dz = 0$ implies $\ipinf = \iminf$, or the following identity holds:
\begin{equation}
\iminf + \ipinf = 2 .
\end{equation}

\begin{proof}
We apply Proposition \eqref{Prop_integrals}. Since $a(\pm \infty) = 0$, the Equations \eqref{first_mass}, \eqref{second_mass} and \eqref{third_mass} simplify to
\begin{align}
\mathscr{A}(+\infty) &= \ipinf \cdot \mathscr{A}(+\infty) + \frac{1+r}{2c} \mathscr{A}(+\infty)^2, \\
\frac{(1+r)}{c} \mathscr{A}(+\infty) &= \iminf -\ipinf.
\end{align}
Either $\mathscr{A}(+\infty) = 0$ implies $\ipinf = \iminf$, or we can divide the first equation by $\mathscr{A}(+\infty)$. Solving the resulting linear system proves the claim.
\end{proof}
\end{corollary}

\section{Asymptotics around the fixed points} \label{ch:fix_point_results}

Let us recall that the Jacobian $D_{(0,0,K)}$ of the ODE System \eqref{3D_Flow_Def} at a fixed point $(0,0,K)$ has eigenvalues
\begin{align}
\lambda_0 = 0, \hspace{1cm}  \lambda_\pm = -\frac{c}{2} \pm \sqrt{\frac{c^2}{4} + K -1}. \label{Eigenvalues_DF}
\end{align}
The existence of $\lambda_0$ implies the existence of a center manifold. In the present case, it locally coincides with the set of fixed points $a=b=0$. This implies that there is no flow along the center manifold, so the asymptotics are fully described by the remaining two linear terms. The calculations are standard and presented in Appendix \ref{ch:Center_manifold}, along with a short review of the underlying theory. We only state the results here. First, regarding the unstable set $S_{-\infty}$, as defined in \eqref{eq_unstable_set}:

\begin{theorem}[Unstable set]
\label{Source_Theorem}
For $\iminf > 1$, the point $(a,\apr,i) = (0,0,\iminf)$ is an unstable fixed point of Dynamics \eqref{3D_Flow_Def}. Locally, there exists a smooth unstable manifold of dimension one. Its restriction to $\{a \geq 0\}$ is the unique trajectory that emerges from the fixed point such that $a(\ta), i(\ta) >0$ as $\ta \rightarrow - \infty$. It has the following properties:
\begin{equation}
\begin{aligned}
&\bullet \lim_{\ta \rightarrow -\infty} a(\ta) = 0, \\
&\bullet \lim_{\ta \rightarrow -\infty} i(\ta) = \iminf, \\
&\bullet b(\ta) >0, \text{ } b'(\ta) >0 , \text{ } i'(\ta) < 0 \quad \text{as } \ta \rightarrow - \infty.
\end{aligned}
\end{equation}
\begin{proof}
By choice of $\iminf>1$, the eigenvalue $\lambda_+$ is positive, whereas $\lambda_-$ is negative. Denote by $u,v,w$ the coordinates in the system of eigenvectors $e_0,e_+,e_-$ of the Jacobian at the fixed, where the fixed point is shifted to the origin. This transformation is done explicitly Lemma \ref{NORMAL_FORM}. By Theorem \ref{CM_thm_main}, the dynamics in an open neighborhood around the fixed point are equivalent to
\begin{align}
\begin{aligned}
u' &= 0, \\
v' &= \lambda_+ \, v, \\
w' &= \lambda_- \, w.
\end{aligned} \label{Dynamics_Normal}
\end{align}
Hence, there is a stable and an unstable manifold, each of dimension one. The eigenvector $e_+$ describes the asymptotic direction of the unstable manifold, in coordinates $a,b,i$ it is given by
\begin{align}
e_+ =
\begin{pmatrix}
-\lambda_-\\
\iminf-1\\
\frac{1}{c}(r + \iminf) \cdot \frac{\lambda_-}{\lambda_+} \\
\end  {pmatrix}. 
\end{align}
Since $\lambda_-<0$ and $\lambda_+ >0$, asymptotically along the branch of the unstable manifold in direction $e_+$, where $a>0$ and $b>0$, it also holds that $b'=\lambda_+ \, (\iminf -1) >0$ and that $ci' = \lambda_- \, (r+\iminf)  <0$.
\end{proof}
\end{theorem}

Next, we prove Lyapunov-stability of the points in $S_{+\infty}$, defined in \eqref{eq_stable_set}. Figure \ref{I0_fixed_phase_space} shows how the phase lines converge to $(0,0)$ in the $(a,b)$-plane. For technical reasons, we require that $\lambda_+ \neq \lambda_-$. Later, we deal with this degenerate case via a continuity argument.

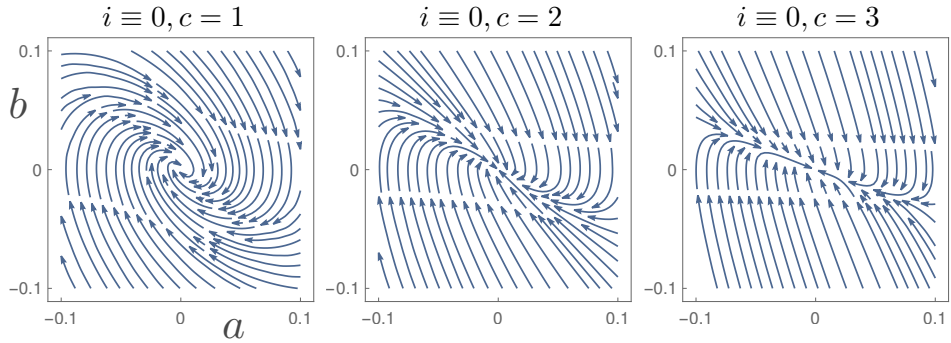
\begin{figure}[h]
\centering
\vspace{0.6cm}
  		\begin{minipage}[c]{0.32\textwidth}
  \begin{picture}(100,100)
	\put(0,0){\includegraphics[width=\textwidth]{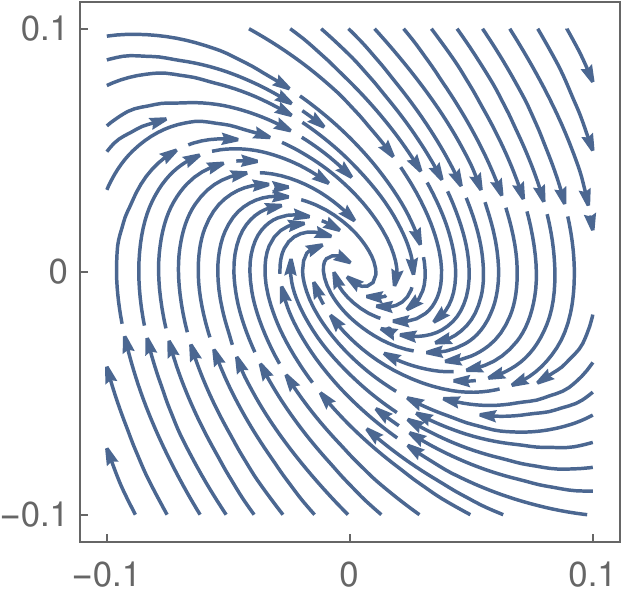}}
	\put(0,80){\scalebox{1.5}{${\color{mygray}b}$}}
	\put(80,-3){\scalebox{1.5}{${\color{mygray}a}$}}
	\put(35,115){$i \equiv 0, c=1$}
\end{picture}
 		\end{minipage}
  		\begin{minipage}[c]{0.32\textwidth}
  \begin{picture}(100,100)
	\put(0,0){\includegraphics[width=\textwidth]{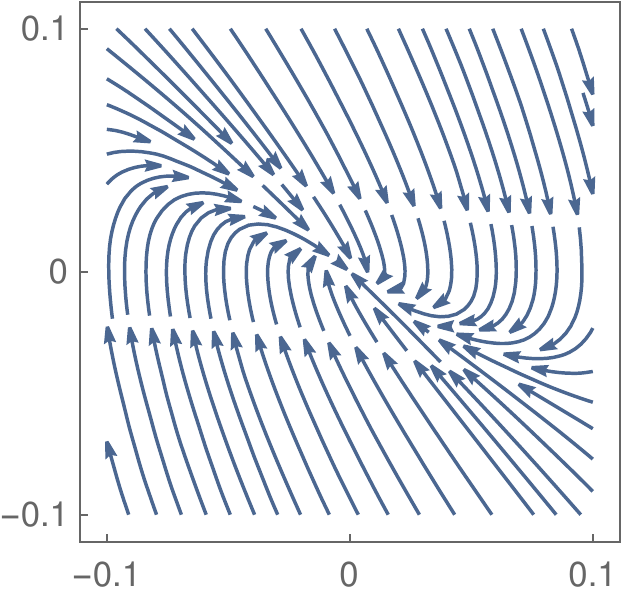}}
	\put(35,115){$i \equiv 0, c=2$}
\end{picture}
 		\end{minipage}
  		\begin{minipage}[c]{0.32\textwidth}
  \begin{picture}(100,100)
	\put(0,0){\includegraphics[width=\textwidth]{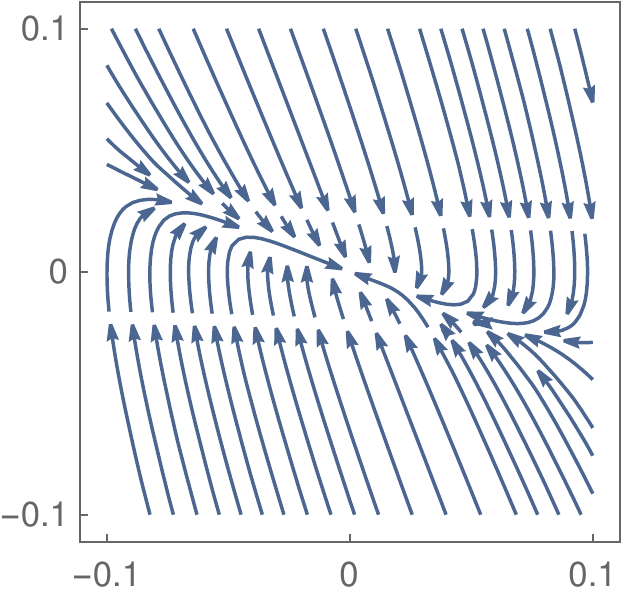}}
	\put(35,115){$i \equiv 0, c=3$}
\end{picture}
 		\end{minipage}
		\vspace{0.2cm}
 		\caption{Phase portrait of $(a,b)$ of the Wave Eq. \eqref{3D_Flow_Def} if we impose a fixed value of $i(\ta) = 0$, see also Section \ref{ch:Attractor}. The choices of $c$ change the type of convergence towards the origin: spiraling for $c=1$, one stable manifold with algebraic index $2$ for $c=2$, two stable manifolds for $c=3$.}\label{I0_fixed_phase_space}
\end{figure}

\begin{theorem}[Stable set]
\label{Sink_Theorem}
For all $c > 0$ and $\ipinf \in [\oc,1)$, such that $\ipinf > c^2/4 -1$, the point $(a,\apr,i)=(0,0,\ipinf) $ is Lyapunov stable under Dynamics \eqref{3D_Flow_Def}. In a small neighborhood, $(a,b) \rightarrow (0,0)$ exponentially fast.
\begin{proof}
By choices of $c$ and $\ipinf$, both non-zero eigenvalues \eqref{Eigenvalues_DF} of the Jacobian are real-valued and negative and it holds that $\lambda_+ \neq \lambda_-$. As before, denote by $u,v,w$ the coordinates in the system of eigenvectors $e_0, e_+, e_-$ of the Jacobian at the fixed point, see Lemma \ref{NORMAL_FORM}. By Theorem \ref{CM_thm_main}, the dynamics of the system in a neighborhood around the fixed point are equivalent to
\begin{align}
\begin{aligned}
u' &= 0, \\
v' &= \lambda_+ \, v, \\
w' &= \lambda_- \, w.
\end{aligned} \label{Dynamics_Normal2}
\end{align}
Take some small enough initial data $(\epsilon_u, \epsilon_v, \epsilon_w)$: in view of Eq. \eqref{Dynamics_Normal2}, $\epsilon_u$ does not vanish, but also does not propagate, whereas $v$ and $w$ converge to zero exponentially fast. Since $a$ and $b$ are represented in terms of $v$ and $w$, see \eqref{Eigenvalues_Normal}, they vanish exponentially fast.
\end{proof}
\end{theorem}

\begin{proposition} \label{Nec_oscillation_convergence}
Let $c > 0$ and $\ipinf < \oc $. There is no non-negative and non-constant traveling wave that converges to $(a,b,i)=(0,0,\ipinf)$ as $\ta \rightarrow +\infty$.
\begin{proof}
As in the previous Theorem, the asymptotic behavior of Eq. \eqref{3D_Flow_Def} around the limiting fixed point $(a,b,i) = (0,0,\ipinf)$ is described by the linear System \eqref{Dynamics_Normal2}. But now, since $\ipinf < \oc = \max\{0, 1- c^2/4\}  $, either $\ipinf < 0$ or both eigenvalues $\lambda_{\pm}$ have a non-vanishing imaginary part and thus, $v$ and $w$ spiral. Since $a$ and $b$ are represented in terms of $v$ and $w$, see \eqref{Eigenvalues_Normal}, any trajectory that converges to $(0,0,\ipinf)$ can not stay non-negative in its $a$-component.
\end{proof}
\end{proposition}

\section{Attractor of a lower-dimensional sub-system} \label{ch:Attractor}

\subsection{Construction and result}

We begin our search for a non-negative attractor of $S_{+\infty}$ in an easier setting: we fix $i(\ta)=i=const.$ and investigate the two-dimensional sub-system in the remaining coordinates. To separate it from the full system, we write it as $\bar{a}(\ta),\bar{\apr}(\ta)$. For this system, we prove the existence of a suitable attractor. This set will be denoted as $T_c(i)$, to emphasize that it depends on the chosen value of $i$, which will be constant only in this section. The flow of the sub-system and the region $T_c(i)$ are drawn in Figure \ref{fig:inv_triangles}.

\begin{definition}[Two-dimensional sub-system] \label{def:2dsyst}
For $c > 0$ and $i \in [ \oc ,1)$, denote by $\bar{a}(\ta),\bar{\apr}(\ta)$ the two-dimensional flow defined by
\begin{align}
\begin{aligned}
\bar{a}' &= \bar{b}, \\
\bar{b}' &= \bar{a}(\bar{a}+i-1) -c\bar{\apr},
\end{aligned} \label{2D_System_Def}
\end{align}
which results from the Wave System \eqref{3D_Flow_Def} by fixing $i(z) = i $.
\end{definition}

There are only two fixed points of \eqref{2D_System_Def}, $(\bar{a},\bar{\apr})=(0,0)$ and $(\bar{a},\bar{\apr})=(1-i,0)$. We denote the eigenvalues and eigenvectors of the Jacobian at $(0,0)$ as
\begin{align}
\lambda_{\pm}(i) & := - \frac{c}{2} \pm \sqrt{\frac{c^2}{4} + i-1},
\qquad  l_\pm(i) : = \begin{pmatrix}
\lambda_\mp \\ 1-i
\end{pmatrix}. \label{Eigen_2d_00}
\end{align}
It holds that $\lambda_-(i) \leq \lambda_+(i) <0$, we see that $(0,0)$ is a stable fixed point. Moreover, for $i > \oc$, it holds that $\lambda_-(i) \neq \lambda_+(i)$. Note that $\lambda_\pm$ are identical to the eigenvalues of the full system around the fixed point $(0,0,i)$, see \eqref{Eigenvalues_DF}. The eigenvectors $l_{\pm}$ are the projections of the corresponding three-dimensional eigenvectors into the $(a,b)$-plane.

The Jacobian at $(1-i,0)$ has eigenvalues and eigenvectors
\begin{align}
\beta_{\pm}(i) & := - \frac{c}{2} \pm \sqrt{\frac{c^2}{4} + 1-i}, \qquad
r_\pm(i) := \begin{pmatrix}
-\beta_\mp \\ 1-i
\end{pmatrix}, \label{Eigen_2d_1-i}
\end{align}
and it holds that $\beta_-(i) < 0 < \beta_+(i)$. These have no direct correspondence to the three-dimensional system.

We now define the region $T_c(i)$. It is a triangle, spanned by the two fixed points $(0,0)$ and $(1-i,0)$ and two adjacent eigenvectors:

\begin{figure}[h]
\vspace{0.7cm}
 		\centering
 		\begin{minipage}[c]{0.32\textwidth}
\begin{picture}(100,100)
	\put(0,0){
 	\begin{tikzpicture}
    \node[anchor=south west,inner sep=0] (image) at (0,0) {\includegraphics[width=\textwidth]{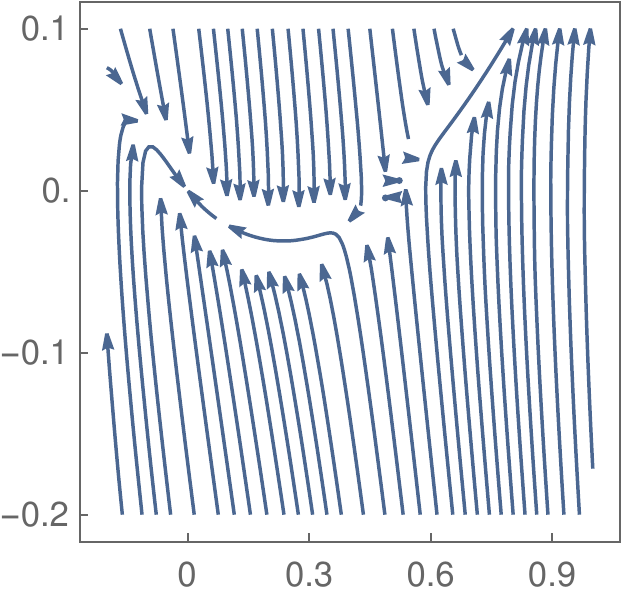}};
    \coordinate (a) at (1.2,2.68);
    \coordinate (b) at (2.55,2.68);
    \coordinate (c) at (1.8, 1.8);
  \draw[line width=0.6 mm,color=RedOrange] (a)--(b)--(c)--(a)--(b);
\pic[draw=black, line width=1.2 pt, angle eccentricity=1.5, angle radius=0.4cm]{angle=a--b--c}; 
\pic[draw=black, line width=1.2 pt, angle eccentricity=1.5, angle radius=0.4cm]{angle=c--a--b};
\filldraw [fill=white, draw=black] (0.75,1.85) rectangle (1.3,2.4);
\filldraw [fill=white, draw=black] (2.45,1.85) rectangle (3.0,2.4);
\end{tikzpicture} }
	\put(3,55){\scalebox{1.5}{${\color{mygray}\bar{b}}$}}
	\put(67,-4){\scalebox{1.5}{${\color{mygray}\bar{a}}$}}
	\put(30,115){$i=0.5, c=2$}
	\put(74,58){\scalebox{1.5}{$\gamma_r$}}
	\put(26,58){\scalebox{1.5}{$\gamma_l$}}
\end{picture} 		
 		\end{minipage}
 		\begin{minipage}[c]{0.32\textwidth}
\begin{picture}(100,100)
	\put(0,0){ 
 	\begin{tikzpicture}
    \node[anchor=south west,inner sep=0] (image) at (0,0) {\includegraphics[width=\textwidth]{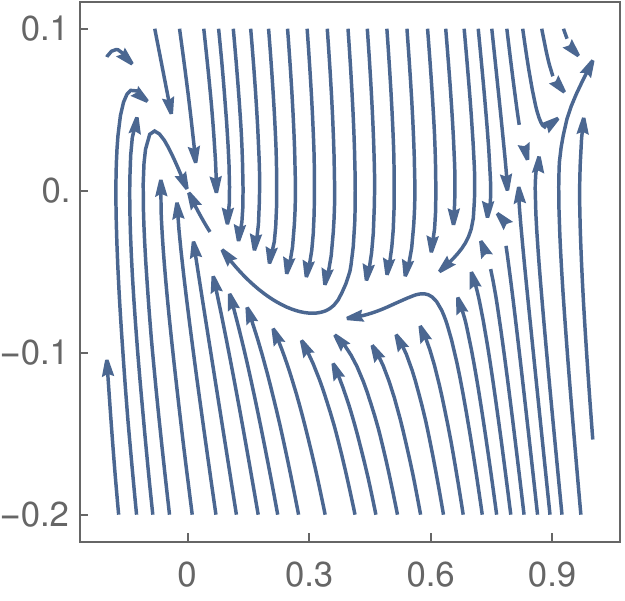}};
\coordinate (a) at (2.0,0.9);
    \coordinate (b) at (1.2, 2.68);
    \coordinate (c) at (3.3,2.68);
  \draw[line width=0.6 mm,color=RedOrange] (a)--(b)--(c)--(a)--(b);
\pic[draw=black, line width=1.2 pt, angle eccentricity=1.5, angle radius=0.4cm]{angle=a--b--c}; 
\pic[draw=black, line width=1.2 pt, angle eccentricity=1.5, angle radius=0.4cm]{angle=b--c--a};
\filldraw [fill=white, draw=black] (0.75,1.7) rectangle (1.3,2.25);
\filldraw [fill=white, draw=black] (3.05,1.7) rectangle (3.6,2.25);
\end{tikzpicture} }
	\put(30,115){$i=0.2, c=2$}
	\put(91,54){\scalebox{1.5}{$\gamma_r$}}
	\put(26,54){\scalebox{1.5}{$\gamma_l$}}
\end{picture} 		
 		\end{minipage}
 		\begin{minipage}[c]{0.32\textwidth}
\begin{picture}(100,100)
	\put(0,-1){\includegraphics[width=1.01\textwidth]{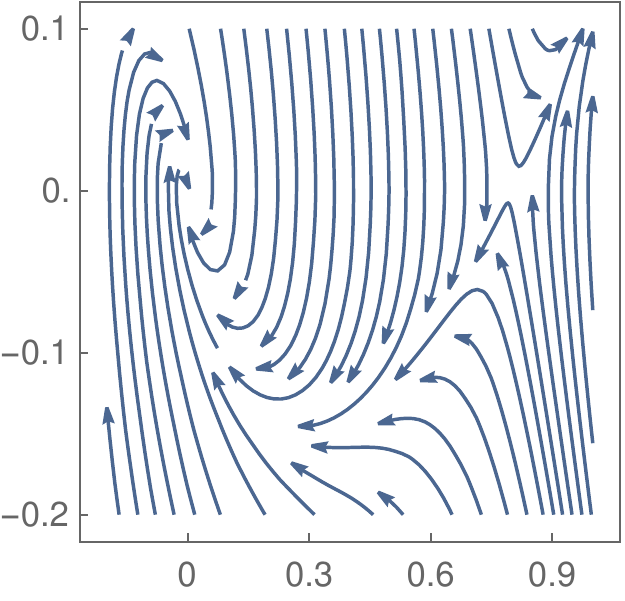}}
	\put(30,115){$i=0.2, c=1$}
\end{picture} 		
 		\end{minipage}
 		\vspace{0.2cm}
 		\caption{The phase plot of $(\bar{a},\bar{\apr})$ following Eq. \eqref{2D_System_Def}, displayed for several values of $i$ and $c$. The only two fixed points are $(0,0)$ and $(1-i,0)$. For $i \geq \oc$, the orange triangles $T_c(i)$ are invariant regions of Dynamics \eqref{2D_System_Def}, see Prop. \ref{single_inv_triangle}. They increase in $-i$: the point $(1-i,0)$ moves to the right and the two internal angles $\gamma_l(i)$ and $\gamma_r(i)$ increase. In the third case, $i < \oc$ implies that the system spirals around $(0,0)$ while converging.} 
 		\label{fig:inv_triangles}
\end{figure}
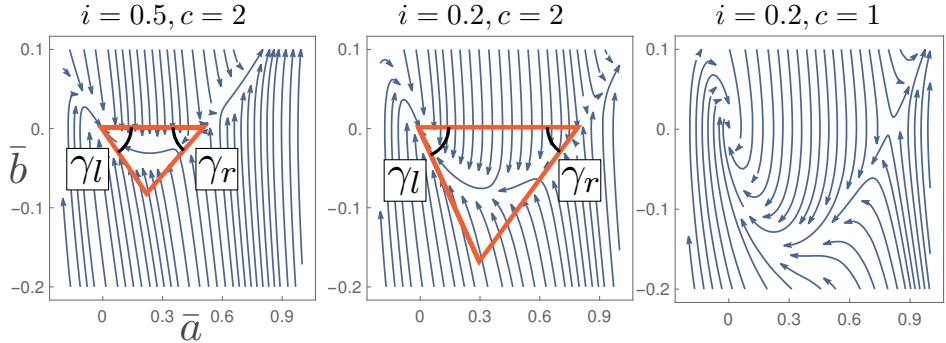

\begin{definition}[The triangle $T_c(i)$] \label{def:triangle}
For $c > 0$ and $i \in [ \oc ,1)$, let $T_c(i)$ be the convex hull of the three points $(0,0), (1-i,0)$ and $ C(i)$. Here, the point $C(i)$ is the unique intersection of the two half-lines
\begin{align}
\begin{aligned}
\Big \{\begin{pmatrix}
0 \\0 \end{pmatrix} - p \cdot l_+(i) \, \Big | \, p  \geq 0 \Big \} \quad \text{and} \quad  \Big \{ \begin{pmatrix}
1-i \\0 \end{pmatrix} - q \cdot r_+(i) \, \Big | \, q \geq 0 \Big \},
\end{aligned}
\end{align}
with $l_+(i)$ and $r_+(i)$ defined in \eqref{Eigen_2d_00} and \eqref{Eigen_2d_1-i}. We denote the internal angles of $T_c(i)$ at $(0,0)$ and $(1-i,0)$ as $\gamma_l(i)$ and $\gamma_r(i)$, respectively.
\end{definition}

Visually, it can easily be seen in Figure \ref{fig:inv_triangles} that the set $T_c(i)$ is invariant under Dynamics \eqref{2D_System_Def}: the flow at the boundary of $T_c(i)$ points inwards. The detailed computations are presented in Section \ref{ch:subsystem_computations}.

\begin{proposition}[Invariant region of the reduced system] \label{single_inv_triangle}
The set $T_c(i)$ is an invariant region of Dynamics \eqref{2D_System_Def}. If $(\bar{a}_0, \bar{b}_0) \in T_c(i)$, then
\begin{align}
\bar{a}(\ta), \bar{b}(\ta) \in T_c(i) \qquad  \forall \ta \geq 0.
\end{align}
It holds that $\bar{a} \geq 0$ and $\bar{b} \leq 0$ within $T_c(i)$. Hence, if $(\bar{a}_0, \bar{b}_0) \neq (1-i,0)$, then $\bar{a}(z)$ converges to $0$ monotonically as $z \rightarrow + \infty$.
\end{proposition}

For any non-negative solution of the full Wave System \eqref{3D_Flow_Def}, it holds that $ i' \leq 0$. Thus, we are interested in how $T_c(i)$ changes when $i$ decreases:

\begin{proposition}[Nested invariant regions] \label{Nested_Triangles}
For a fixed $c>0$, the set $T_c(i)$ is increasing in $-i,i \in [\oc,1)$. Thus, $T_c(i) \subseteq T_c(\oc)$ for all $i \in [\oc,1)$.
\end{proposition}

This Proposition is proven by an easy geometric argument, again take a look at Figure \ref{fig:inv_triangles}: when $i$ decreases, the point $(1-i,0)$ moves to the right and the two internal angles $\gamma_l(i)$ and $\gamma_r(i)$ increase. Any reader who is not interested in the computational details may proceed with Section \ref{ch:full_system_attractor}, where we investigate the full system.

\subsection{Invariance and monotonicity of $\boldsymbol{T_c(i)}$}
\label{ch:subsystem_computations}
We analyze the $(\bar{a},\bar{\apr})$-system and the set $T_c(i)$ in detail. We prove that the Flow \eqref{2D_System_Def} at the boundary of $T_c(i)$ points inwards, and that the sets $T_c(i)$ are increasing in $-i$. For not getting lost in the following fiddly computations, always keep Figure \ref{fig:inv_triangles} in mind. We first examine the eigenvector $l_+(i)$ at the fixed point $(\bar{a}, \bar{b}) = (0,0)$:

\begin{lemma} \label{lem:apex_l}
Let $1>i > \oc$, and let $l_+(i) =( \lambda_-(i), 1-i)$ be defined as in \eqref{Eigen_2d_00}. The quotient of the absolute values of the $\bar{\apr}$-component and $\bar{a}$-component of $l_+(i)$ is increasing in $-i$.
\begin{proof}
The claim is equivalent to
\begin{align}
\frac{d}{di} \, \frac{ | \lambda_-(i) | } { 1-i}  > 0.
\end{align}
Recall that $\alpha_ - (i) = -c/2 - \sqrt{c^2/4 + i -1}<0$. A computation reveals that
\begin{align}
\begin{aligned}
\frac{d}{d i } \, \frac{ | \lambda_-( i ) | } { 1- i } &= - \frac{d}{d i } \, \frac{ \lambda_- ( i ) }{1- i }  \\ 
&= \frac{ \frac{1-i}{2 \sqrt{c^2/4 + i-1}} + \frac{c}{2} + \sqrt{c^2/4 + i-1} }{(1-i)^2}  \\
&= \frac{ 
1-i + c \sqrt{c^2/4 + i-1} +2(\frac{c^2}{4} + i-1)}{2 (1-i)^2 \sqrt{c^2/4 + i-1}} > 0,
\end{aligned}
\end{align}
the last inequality holds since $\oc < i $ implies $i > 1-c^2/4$.
\end{proof}
\end{lemma}

Now, since $\gamma_l(i)$ is the angle between the two vectors $(0,1)$ and $-l_+(i)$, the previous Lemma directly implies
\begin{corollary}
\label{cor:incr_cone_1}
Let $c > 0$ and $\oc \leq i_1 < i_2 <1 $. It holds that $\gamma_l(i_1) > \gamma_l(i_2)$, the angle $\gamma_l(i)$ is increasing in $-i$.
\end{corollary}

For the invariance of $T_c(i)$, we need

\begin{lemma} \label{Lemma_cone1}
For any $p>0$, the Flow \eqref{2D_System_Def} at the point $(\bar{a}, \bar{\apr}) = -p \cdot l_+(i)$ points inwards $T_c(i)$.
\begin{proof}
Let $L_{inw} := \big(1-i, - \lambda_-(i) \big)$ be orthogonal to $l_+(i)$ and point inwards $T_c(i)$. The claim of the Lemma is now equivalent to
\begin{align}
\Big\langle \begin{pmatrix} \bar{a}' \\ \bar{b}'
\end{pmatrix} , L_{inw} \Big\rangle > 0.
\end{align}
Let $(\bar{a}, \bar{\apr}) = -p \cdot l_+(i)$. First compute the Flow \eqref{2D_System_Def}:
\begin{align}
\begin{aligned}
\begin{pmatrix} \bar{a}' \\ \bar{\apr}' \end{pmatrix} 
&= \begin{pmatrix}
-p(1-i) \\
-p \lambda_-(-p \lambda_- + i-1) - c \big( -p(1-i) \big)
\end{pmatrix} \\
&= p \begin{pmatrix}
-(1-i) \\
p \lambda_-^2 + (1-i)(\lambda_- +c)
\end{pmatrix} .
\end{aligned}
\end{align}
Its part in direction $L_{inw}$ is given by
\begin{align}
\begin{aligned}
\Big\langle  \begin{pmatrix} \bar{a}' \\ \bar{\apr}'
\end{pmatrix} , L_{inw} \Big\rangle &= -p \Big [
(1-i) \big(1-i + \lambda_-(\lambda_- +c ) \big) + p \lambda_-^3
\Big] \\
&= -p \Big [
(1-i) \big(1-i + i -1 \big) + p \lambda_-^3
\Big] = -p^2 \lambda_-^3 > 0.
\end{aligned}
\end{align}
\end{proof}
\end{lemma}

For the fixed point $(\bar{a}, \bar{b}) = (1-i,0)$, we get similar results concerning its unstable eigenvector $r_+(i)$:

\begin{lemma} \label{lem:apex_r}
Let $1>i>i_c$, and let $r_+(i) =( -\beta_-(i), 1-i)$ be defined as in \eqref{Eigen_2d_1-i}. The quotient of the absolute values of the $\bar{\apr}$-component and $\bar{a}$-component of $r_+(i)$ is increasing in $-i$:
\begin{align}
\frac{d}{di} \, \frac{ 1-i}{ | \beta_-(i) | } < 0.
\end{align}
\begin{proof}
Recall that $\beta_-(i) = -c/2 -\sqrt{c^2/4 + 1-i}$. A computation reveals that
\begin{align}
\frac{d}{di} \, \frac{ 1-i}{ | \beta_-(i) | }  &=  - \frac{d}{di} \frac{ 1-i }{\beta_-(i)} = - \frac{ - \beta_- + \frac{1-i}{2 \sqrt{c^2/4 + 1 -i}} }{\beta_-^2}  <0
\end{align}
\end{proof}
\end{lemma}

Now, since $\gamma_r(i)$ is the angle between the two vectors $(0,-1)$ and $-r_+(i)$, the previous Lemma implies

\begin{corollary}
\label{cor:incr_cone_2}
Let $c > 0$ and $\oc \leq i_1 < i_2 <1 $. It holds that $\gamma_r(i_1) > \gamma_r(i_2)$, the angle $\gamma_l(i)$ is increasing in $-i$.
\end{corollary}

For the invariance of $T_c(i)$, we need

\begin{lemma} \label{Lemma_cone2}
For any $p>0$, the Flow \eqref{2D_System_Def} at the point $(\bar{a}, \bar{\apr}) = (1-i,0) -p \cdot r_+(i)$ points inwards $T_c(i)$.
\begin{proof}
Let $R_{inw} := \big(i-1, - \beta_-(i) \big)$ be orthogonal to $r_+(i)$ and point inwards $T_c(i)$. The claim of the Lemma is now equivalent to
\begin{align}
\Big\langle \begin{pmatrix}
\bar{a}' \\ \bar{\apr}' \end{pmatrix}, R_{inw} \Big\rangle > 0.
\end{align}
Let $p>0$. We compute the Flow \eqref{2D_System_Def} at
\begin{equation}
\begin{aligned}
\begin{pmatrix}
\bar{a} \\ \bar{\apr} \end{pmatrix} &=  \begin{pmatrix}
1-0 \\0
\end{pmatrix} - p \cdot r_+(i) = 
\begin{pmatrix}
1-i + p \beta_- \\
p(i-1)
\end{pmatrix}: \\
\begin{pmatrix}
\bar{a}' \\ \bar{\apr}' \end{pmatrix} &= \begin{pmatrix}
p(i-1) \\ (1-i + p \beta_-) ( 1-i + p \beta_- + i -1) - cp(i-1)
\end{pmatrix}   \\
&= p \begin{pmatrix}
i-1 \\ \beta_- (1-i + p\beta_-) - c(i-1)
\end{pmatrix}.
\end{aligned}
\end{equation}
Its part in direction $R_{inw}$ is given by
\begin{equation}
\Big\langle \begin{pmatrix}
\bar{a}' \\ \bar{\apr}' \end{pmatrix}, R_{inw} \Big\rangle = p \Big [ 
(i-1)^2 - \beta_- \big[ 
\beta_- (1-i) + p \beta_-^2 -c(i-1)
\big ]
\Big ].
\end{equation}
Since $\beta_- <0$, it follows that $ - p^2 \beta_-^3 >0$. Since $p>0$, the proof is complete if we can show that
\begin{align}
(1-i)^2 - \beta_- \big[ \beta_-(1-i) +c(1-i) \big ] & \geq 0. \qquad
\intertext{After dividing by $(1-i)>0$ and rearranging, this is equivalent to}
1-i -c \beta_- & \geq \beta_-^2.
\end{align}
This is in fact an equality, since $\beta_-(i) = -c/2 + \sqrt{c^2/4 +1-i}$.
\end{proof}
\end{lemma}

Considering the invariance of $T_c(i)$, we conclude the
\begin{proof}[\textbf{Proof of Proposition \ref{single_inv_triangle}}]
We need to show that the Flow \eqref{2D_System_Def} at the boundary of $T_c(i)$ points inwards. The Lemmas \ref{Lemma_cone1} and \ref{Lemma_cone2} handle the left and right edge of $T_c(i)$, see again Figure \ref{fig:inv_triangles}. For the third edge, we need to consider points of type $(\bar{a},0)$, where $0 < \bar{a} < 1-i$. The derivative at $(\bar{a},0)$ is given by $ (0, \bar{a}(\bar{a} +i -1) )$. Its $\bar{\apr}$-component is negative, hence it points inwards $T_c(i)$. The only points on the boundary of $T_c(i)$ where the flow does not point strictly inwards are the two fixed points $(0,0)$ and $(1-i,0)$.
\end{proof}

Considering the monotonicity of $T_c(i)$, we conclude the
\begin{proof}[\textbf{Proof of Proposition \ref{Nested_Triangles}}]
Let $c>0$. The point $(0,1-i)$ moves to the right as $i$ decreases. Further, we have shown that the two internal angles $\gamma_l(i), \gamma_r(i)$ are increasing in $-i$, and so does $T_c(i)$.
\end{proof}

\section{Attractor of the full system}
\label{ch:full_system_attractor}

We now analyze solutions of the full Wave System \eqref{3D_Flow_Def} under initial condition $(a,b,i) = (a_0,0,i_0)$, such that $(a_0,0) \in T_c(i_0)$ as defined in the previous Section. In Section \ref{ch:elevate_result}, we apply the results about the two-dimensional subsystem to the full system. Theorem \ref{Triangles_3D_Thm} states that as long as $i(\ta) \geq \oc$, the $(a,\apr)$-components of the full system stay within the triangle $T_c(\oc)$. Thus, it suffices to control $i(\ta) \geq \oc$, which we do in two steps.

In Section \ref{ch:comparison_theorem}, we prove that $i(\ta) \geq \oc$ for sufficiently small initial values $0 \leq a_0 \ll 1$. Some rough bounds do the trick. This result is refined in Section \ref{ch:extending_attractor}: the Lyapunov-stability of the limiting point at $\ta  = + \infty$ implies that the entire trajectory including its limit is continuous in initial data. Carefully increasing $a_0$, we increase the known attractor of the stable set $S_{+\infty}$, resulting in Theorem \ref{a_0_I_limit}. This procedure is sketched in Figure \ref{fig:a0_infty_extend}.\\

\textbf{Assumption}: If not explicitly stated otherwise, we will use the following setup over the entire Section \ref{ch:full_system_attractor}: For $c>0$, let $i_0 \in [ \oc ,1)$ and $ a_0 \in [0,1-i_0]$, so that $(a_0,0) \in T_c(i_0)$. Let $a(\ta),\apr(\ta),i(\ta)|_{z \geq 0 }$ be the solution of the Wave Eq. \eqref{3D_Flow_Def} with initial values $(a_0, 0, i_0)$.

\subsection{Invariant region of the full system}
\label{ch:elevate_result}

\begin{theorem}[Invariant region of the full system] \label{Triangles_3D_Thm}
Assume that $i(\ta) \geq \oc$ for all $\ta \in [0, \infty)$. We can control the two remaining coordinates $a(z),b(z)$ of the wave. It holds that
\begin{align}
 a(\ta), b(\ta) \, \in \, T_c(\oc) \qquad  \forall \ta \in [0, \infty).
\end{align}
\end{theorem}
Within $T_c(\oc)$, $a\geq 0$ and $b \leq 0$. Notice that while $a,i \geq 0$, it holds that $ci' = -a(a+i+r) \leq 0$. This directly implies the following
\begin{corollary}
Under the assumption that $i(\ta) \geq \oc$ for all $\ta \in [0, \infty)$, the trajectory stays non-negative and converges as $\ta \rightarrow + \infty$:
\begin{align}
\begin{aligned}
a(\ta) & \rightarrow 0, \\
b(\ta) & \rightarrow 0, \\
i(\ta) & \rightarrow \ipinf  \in [ \oc , 1).
\end{aligned}
\end{align}
\end{corollary}
\begin{proof}[Proof of Theorem \ref{Triangles_3D_Thm}] In the full System \eqref{3D_Flow_Def} with coordinates $(a,b,i)$, neither $\apr$ nor $\apr'$ depend on $i'$, but only on $a$ and $i$. Thus, we can easily compare the full system to the two-dimensional System \eqref{2D_System_Def} in coordinates $\bar{a}, \bar{b}$. At a phase-time $\ta$, the two vector fields $ (a',\apr')$ and $ (\bar{a}',\bar{b}')$ for fixed value $i = i(\ta)$ are equal, compare \eqref{3D_Flow_Def} and \eqref{2D_System_Def}.

By Proposition \ref{Nested_Triangles}, this implies that $ (a',\apr')$ points strictly inwards $T_c\big(i(\ta)\big)$. There are two irrelevant exceptions: for $(a,b) = (0,0)$, the system has already reached its limiting state. The point $(a,b) = (1-i_0,0)$ is a fixed point of the reduced, but not of the full system. Since $b=0,b'=a(a+i-1) = 0, b'' = ai'<0$ and $ci'<0$, a Taylor-expansion reveals that $a(\epsilon)$ lies in the interior of $T_c\big( i(\epsilon) \big)$ for small times $\epsilon >0$.

Importantly, the set $T_c\big( i(\ta) \big)$ is not decreasing as a function of $\ta$. Hence, at each phase-time $\ta \geq 0$, the two components $a(\ta), b(\ta)$ can not escape $T_c\big(i(\ta)\big)$. In fact, since $i' \leq 0$ and $T_c(i)$ is increasing in $-i$, the set $T_c\big( i(\ta) \big)$ is increasing in $\ta$, at most up to $T_c(\oc)$. Thus, $a(\ta), b(\ta)$ remain within $T_c(\oc)$ for all $\ta \geq 0$.
\end{proof}

With a very similar argument, we can determine the rate of convergence:
\begin{proposition} \label{Prop:Tail_exp_rate_integrable}
Assume that $i(z) \geq \oc$ for all $z \in [0,\infty)$, such that $(a,b,i) \rightarrow (0,0,\ipinf)$ as $z \rightarrow +\infty$ for some $\ipinf \in [\oc,1)$. If $\ipinf > c^2/4 -1$, then convergence is exponentially fast with rate
\begin{align}
\mu_{+ \infty} = - \frac{c}{2} + \sqrt{ \frac{c^2}{4} + \ipinf -1} < 0.
\end{align}
Further, if $(c/2)^2 - \ipinf -1 = 0$, which can only happen if $\ipinf = \oc$, then the system converges sub-exponentially fast. As $z \rightarrow + \infty$, the distance to the limit is of order
\begin{align}
z \cdot e^{ - \frac{c}{2} z}.
\end{align}
\begin{proof}
In the case $\ipinf \in (\oc,1)$, all eigenvalues of the limit are simple, we refer to Section \ref{ch:fix_point_results}. The system converges exponentially fast, as shown in Theorem \ref{Sink_Theorem}. It remains to determine the rate of convergence. The two candidates are $\lambda_\pm = - \frac{c}{2} \pm \sqrt{ \frac{c^2}{4} + \ipinf -1}$. Corresponding to $\lambda_\pm$, the projections of the eigenvectors into the $(a,b)$-plane are given by
\begin{align}
l_\pm : = \begin{pmatrix}
\lambda_\mp \\ 1- \ipinf
\end{pmatrix}.
\end{align}
We know that $a(z), b(z) \in T_c(\ipinf)$ for all $z \geq 0$. At $(0,0)$, the triangle $T_c(\ipinf)$ is bounded by the line $ - l_+$, see Def. \ref{def:triangle} and Figure \ref{fig:inv_triangles}. Since $0 > \lambda_+ > \lambda_-$, the direction of $l_-$ is steeper than that of $l_+$, such that the line $\{q \cdot l_- \text{ | } q \in \mathds{R} \}$ lies outside $T_c(\ipinf)$ for all $q \neq 0$. Thus, the two components $a(z), b(z)$ cannot converge towards $(0,0)$ along $-l_-$. But since they converge exponentially fast, the only possible remaining rate of convergence is $\lambda_+$.

For the case $\lambda_+ = \lambda_- = - c/2$, we do not have a complete description of the asymptotics around the fixed point. However, under the assumption that the system stays non-negative and converges, it converges along a stable manifold along which $(a,b,i)'= -\frac{c}{2} \cdot (a,b,i)$ asymptotically. This will be proven in Theorem \ref{a_0_I_limit} (which does not rely on this Proposition). Thus, it suffices to describe the dynamics on this manifold. The eigenvalue $-c/2$ has algebraic multiplicity $2$, but only a single linearly independent eigenvector. In this setting, the subspace that corresponds to the eigenvalue is spanned by one eigenvector and one generalized eigenvector. It is well-known that this results in sub-exponential convergence, cf. chapter 9 in \cite{Boyce_Prima}.
\end{proof}
\end{proposition}

\subsection{A small attractor}
\label{ch:comparison_theorem}

In view of the previous paragraph, convergence and non-negativeness follow if we can show that $i(z) \geq \oc$ for all $z \geq 0$. If we choose $a_0$ small enough, some rough bounds do the trick. We control the total mass of active particles via

\begin{lemma} \label{lem_Small_Attractor}
Fix $c>0, \, i_0 \in (\oc,1)$ and let $a_0 \in [0,(1-i_0)/2]$. Under the assumption that $i(s) \geq \oc$ is true for all $s \in [0,\ta]$, there exists a finite constant $L(c,i_0) \geq 0$, such that the following bound holds for all $ s \in [0,\ta]$:
\begin{align}
\int_0^s a(t) \, dt \leq \frac{c a_0-b(\ta)}{1-(i_0+a_0)} \leq L \cdot a_0.
\end{align}
\begin{proof} With the help of Theorem \ref{Triangles_3D_Thm}, we can use that $a(s), b(s) \in T_c(\oc)$ for all $ s \in [0,\ta]$. We integrate $  a(s)\cdot \big[1-i(s) \big] = a^2(s) - b'(s) - c b(s) $ and use that $b(0) = 0$:
\begin{align}
 \quad \int_0^\ta a(s) \cdot \big[  1 - i(s) \big] \, ds  = c a_0 - c a(\ta) - b(\ta) + \int_0^\ta a^2(s) \, ds.
\end{align}
By monotonicity: $1-i_0 \leq 1-i(s)$ and $0 \leq a(s)\leq a_0$. It follows that
\begin{align}
\begin{aligned}
 (1-i_0-a_0) \int_0^\ta a(s) \, ds & \leq ca_0 - ca(z) - b(z) \\
&\leq c a_0 -b(z) \\
\Leftrightarrow \int_0^\ta a(s) \, ds
& \leq \frac{c a_0-b(\ta)}{1-(i_0+a_0)}, \label{lem_ineq_conv}
\end{aligned}
\end{align}
where we need $a_0 + i_0 <1$ to avoid a blow-up, which is true by our choice of $a_0$. It remains to bound $-b(\ta) $. Take a look at the flow in the $(a,b)$-plane in Figure \ref{fig:inv_triangles}. It holds that $a(\ta),b(\ta)$ stay within $T_c(\oc)$. Within the triangle $T_c(\oc)$, it holds for the left inner angle $\gamma_l(\oc)$ that
\begin{align}
\tan \big ( \gamma_l(\oc) \big) \geq \frac{|b|}{|a|}.
\end{align}
Thus, also $-b(\ta) \leq \tan \big ( \gamma_l(\oc) \big) \cdot a(z) \leq L_2 \cdot a_0$.
\end{proof}
\end{lemma}

Since we can bound the total mass of active particles, we can also bound the change of $i(z)$:

\begin{proposition}[Small attractor of $S_{+\infty}$]
\label{Small_Attractor}
Fix $c>0$ and $i \in ( \oc, 1)$. There exists a positive constant $0 < M \ll 1$, such that for all $0 \leq a_0 \leq M$:
\begin{align}
i(\ta)  \geq  \oc \qquad  \forall \ta \geq 0.
\end{align}
Hence, also $ a(\ta),b(\ta) \in T_c(\oc)$ for all $z \geq 0$. The trajectory is non-negative and converges to $S_{+\infty}$ as $\ta \rightarrow +\infty$.
\begin{proof}
As long as $i(\ta) \geq \oc$, it must be that $a(\ta),b(\ta) \in T_c(\oc)$ by Theorem \ref{Triangles_3D_Thm}. Assume there exists finite phase-time $\tau := \inf_{z \geq 0} \{ i(z) < \oc \} $:
\begin{align}
\begin{aligned}
i(\tau) & = i_0 + \int_0^\tau i'(z) \, dz &&= i_0 - \frac{1}{c} \int_0^{\tau} a(s)\big[a(s)+i(s)+r \big] \, ds \\
&&& \geq i_0 - \frac{1}{c}  \int_0^{\tau} a(s)\big[1 +r \big] \, ds,
\end{aligned} \label{ineq:i(t)}
\end{align}
where we used $a(s) + i(s) \leq 1$. For $z \leq \tau$ and $a_0$ sufficiently small, we can apply Lemma \ref{lem_Small_Attractor}. This implies that there is a finite constant $L$, which does not depend on $a_0$, such that
\begin{align}
i(\tau) \geq i_0 - \frac{L}{c}(1+r) \cdot a_0.
\end{align}
The right-hand side is strictly larger than $\oc$ for sufficiently small $a_0$, say $a_0 \leq M$. Thus, there is no such phase-time $\tau$ for $a_0 \leq M$.
\end{proof}
\end{proposition}

\subsection{Extending the attractor}
\label{ch:extending_attractor}

The previous section ended with a condition of type $a_0 \ll 1$, under which the system stays non-negative and converges. However, given $a_0$ and $i_0$ and under the assumption that the system converges, we can explicitly calculate its limit $\ipinf$. For fixed $i_0$, we continuously deform the trajectory while increasing $a_0$ up to some upper bound $a^\ast(i_0)$, as sketched in Figure \ref{fig:a0_infty_extend}.

In what follows, we first assume that the system stays non-negative and converges, analyze its behavior under this assumption, and then verify that this must be true for certain initial data, which leads to Theorem \ref{a_0_I_limit}.

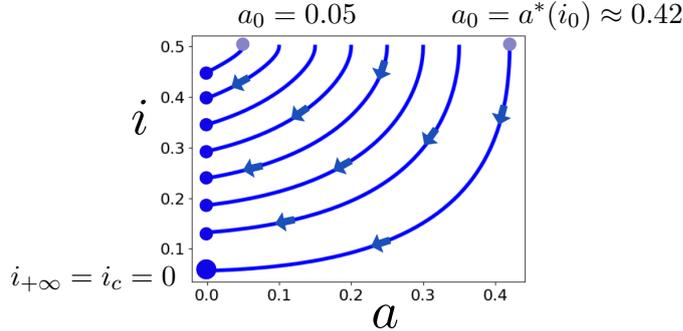
\begin{figure}[h]
\vspace{0.2cm}
 	\centering
 	\begin{minipage}[c]{0.5\textwidth}
  	\begin{picture}(100,100)
	\put(0,0){\includegraphics[width=0.8\textwidth]{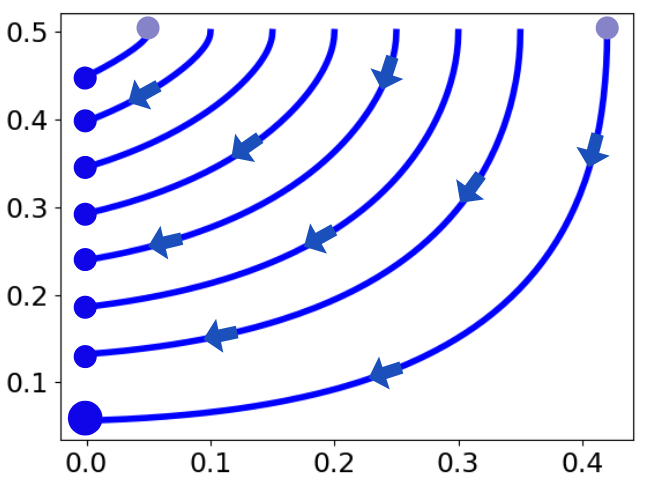}}
	\put(-10,65){\scalebox{1.8}{$i$}}
	\put(-55,8){$\ipinf = \oc = 0 $}
	\put(80,-8){\scalebox{1.8}{$a$}}
	\put(30,107){$a_0 = 0.05$}
	\put(110,107){$a_0 = a^\ast(i_0) \approx 0.42$}
	\end{picture}
	\end{minipage}
	\vspace{0.25cm}
	\caption{Trajectories of $a(\ta),i(\ta)$ of the Wave System \eqref{3D_Flow_Def} for $c = 2, r=0$. Initial values are $b(0) = 0, i_0=0.5$, and $a_0$ such that $a_0 \in [0, a^\ast(i_0) \approx 0.42]$. The upper bound $a^\ast$ is given in Def. \ref{A_ast_def}. Trajectories with such initial data converge and stay non-negative, since $i(z) \geq \oc$.}
	\label{fig:a0_infty_extend}
\end{figure}

We apply Proposition \ref{Prop_integrals} to interrelate the limit $(0,0,\ipinf)$ of the trajectory to its initial data $(a_0,0,i_0)$:
\begin{lemma} \label{I_+infty_formula_lemma}
If $i(z) \geq \oc $, such that the system stays non-negative and converges to $(0,0,\ipinf)$ as $\ta \rightarrow + \infty$, then $\ipinf$ can be written as a function of $a_0$ and $i_0$:
\begin{align}
\ipinf(a_0,i_0) = 1 - \sqrt{ (i_0 + a_0 -1)^2 + \frac{1+r}{c^2}( a_0^2 +2c^2a_0)
}. \label{I_+infty_formula}
\end{align}
The function $\ipinf(a_0,i_0)$ is decreasing in $a_0$, for $a_0 \in [0,1-i_0]$.
\end{lemma}
\begin{proof}
We apply Proposition \ref{Prop_integrals}, and solve the resulting system for $\ipinf$. In the present case, since $a_0 \neq 0$, this results in a quadratic equation with two possible solutions. By positiveness of $a(\ta)$ and $i(\ta)$, $i(\ta)$ is decreasing, so it must be that $\ipinf < 1$, which uniquely determines \eqref{I_+infty_formula}. It can easily be verified that $\frac{d}{d \, a_0} \ipinf(a_0,i_0) \leq 0$. 
\end{proof}

We look for values of $a_0$ that ensure $i_{+\infty} \geq \oc$. Thus, we rearrange \eqref{I_+infty_formula} for $a_0$, set $\ipinf = \oc$, and choose the only possible positive solution of the resulting quadratic equation:
\begin{lemma}
Given $i_0$ and under the assumption that $\ipinf=\oc$, the value of $a_0$ is uniquely determined by
\begin{align}
\begin{aligned}
	\alpha(i_0) &:= \frac{c^2}{1+c^2+r}
	\Bigg\{ -(i_0 +r) \label{a0_I+infty_formula} \\
	& \quad + \sqrt{ (i_0 +r)^2 + \frac{c^2+1+r}{c^2}
	\Big ( (1-\oc)^2 - (1-i_0)^2 \Big )	} \Bigg\}.
\end{aligned}
\end{align}
\end{lemma}

Equation \eqref{a0_I+infty_formula} can be restated as $\ipinf \big ( \alpha(i_0), i_0 \big ) = \oc$, but keep in mind that still have to prove convergence. It can easily be seen that $\alpha(\oc) = 0$. Since we require that $a_0 \in [0, 1 -i_0]$, such that $a_0 \in T_c(i_0)$, this leads to our

\begin{definition}[Upper bound for $a_0$] \label{A_ast_def}
For fixed $c>0$ and, we define
\begin{align}
a^\ast(i_0) := \min \Big\{ \alpha(i_0), 1-i_0 \Big\}, \quad \text{for } i_0 \in [\oc,1).
\end{align}
\end{definition}
This will hold as sharp upper bound for $a_0$, such that the trajectory stays non-negative and converges. Before we state the corresponding theorem, we perform a last check that we are in the correct setup:
\begin{lemma} \label{A_ast_lemma}
Let $i_0 \in [\oc,1)$ and $a_0 \in [0, a^\ast(i_0)]$. If the system stays non-negative and converges to $(0,0,\ipinf)$, then
\begin{align}
\ipinf(a_0,i_0) \, \in \,  [\oc, i_0],
\end{align}
where $\ipinf(a_0,i_0)$ is given as in Lemma \ref{I_+infty_formula_lemma}.
\begin{proof}
It holds that $\ipinf(0,i_0) = i_0$. The claim follows since $\ipinf(a_0,i_0)$ is decreasing in $a_0$ and since $a^\ast(i_0) \leq \alpha(i_0)$, where $\ipinf \big ( \alpha (i_0), i_0 \big ) = \oc$.
\end{proof}
\end{lemma}

Finally, we can remove the training wheels and get rid of the assumption that the system converges and stays non-negative. We state

\begin{theorem} [Attractor of $S_{+\infty}$] \label{a_0_I_limit}
For $r\geq 0, c>0$, let $ i_0 \in [\oc,1)$ and $ a_0 \in \big [0, a^\ast(i_0) \big ]$. Let $a(\ta),b(\ta),i(\ta)$ be the solution of Eq. \eqref{3D_Flow_Def} with initial data $(a_0,0,i_0)$.

It then holds that $a(\ta),i(\ta) \geq 0$ and $i'(\ta),b(\ta) \leq 0 $ for all $\ta \geq 0$. As $\ta \rightarrow +\infty$, $a(\ta)$ and $b(\ta)$ converge to $0$, and $i(\ta)$ converges to
\begin{align}
\ipinf(a_0,i_0) = 1 - \sqrt{ (i_0 + a_0 -1)^2 + \frac{1+r}{c^2}( a_0^2 +2c^2a_0)
} \quad \in \, [\oc,1). \label{eq:attr_big_thm}
\end{align}
The type of convergence is given in Proposition \ref{Prop:Tail_exp_rate_integrable}.
\begin{proof}
\textbf{Notation}: We fix $r,c,i_0 \in (\oc,1)$ and change only $a_0$. If $i_0=\oc$, we must choose $a_0 =0$. For visual clarity, $\Phi_\ta(x)$ is the state of the system at phase-time $\ta$, starting in ${x=(a,b,i)}$. If the limit of a trajectory exists, we denote
\begin{align}
\Phi_{+\infty}(a_0,0,i_0) := \lim_{\ta \rightarrow +\infty} \Phi_\ta ( a_0,0,i_0 ) = (0,0,\ipinf ).
\end{align}

\textbf{Step 1: starting interval}\\
For $a_0$ positive but small enough, Proposition \ref{Small_Attractor} grants that for all $\ta \geq 0$:
\begin{align}
i(\ta) \geq \oc \quad \text{ and } \quad  a(\ta),b(\ta)  \in T_c(\oc),
\end{align}
where $T_c(\oc)$ is a bounded invariant region that contains only points such that $a \geq 0, b \leq 0 $. Thus, $ a(\ta),i(\ta)  \rightarrow (0, \ipinf)$ monotone. With the help of Lemma \ref{I_+infty_formula_lemma}, we can explicitly calculate $\ipinf$ as stated, and Eq. \ref{eq:attr_big_thm} holds on some small non-empty interval $a_0 \in [0,a_u)$.\\

\textbf{Step 2: neighborhood of existing trajectories}\\
Pick some $a_0<a^\ast(i_0)$ for which the statement is already proven. By choice of $a_0$, it holds that $\ipinf > \oc$. Thus, the limit $\Phi_{+\infty}(a_0,0,i_0)$ is Lyapunov stable by our previous analysis of the asymptotics, see Theorem \ref{Sink_Theorem}: for every $\epsilon_\infty  >0$, there exists a $\delta_\infty >0$, such that
\begin{align}
||x-\Phi_{+\infty}(a_0,0,i_0)|| < \delta_\infty \Rightarrow ||\Phi_\ta(x) - \Phi_{+\infty}(a_0,0,i_0)|| < \epsilon_\infty
\end{align}
for all $\ta \in [0, \infty)$. Choose $ \epsilon_\infty \leq \ipinf(a_0,i_0) -\oc $ and $0 < \delta_\infty \leq \epsilon_\infty$. This grants $i(\ta) \geq \oc$ after entering the $\delta_\infty$-neighborhood. Within this attractor, also $a(\ta) \geq 0$ in view of Theorem \ref{Triangles_3D_Thm}, since $i(\ta) \geq \oc $.

Starting in $(a_0,0,i_0)$, we follow the trajectory up to some finite time $\tau$, where it has entered the $\delta_\infty$-neighborhood:
\begin{align} \label{Eq:proof_big_attr1}
||\Phi_{\tau}(a_0,0,i_0) - \Phi_{+\infty}(a_0,0,i_0)|| \leq \frac{\delta_\infty}{2}.
\end{align}
The derivative of the system is locally Lipschitz continuous, and so trajectories $\Phi_\ta(x_0)$ are uniformly continuous in initial data $x_0$ on finite time intervals. This is a classic result, we refer to Chapter 2 of the textbook of Hsieh and Yasutaka \cite{Hsieh_Sibuya_Basic_Theory_ODEs}. There exists some $\delta_0 >0$, s.t. for all $x \in \mathds{R}^3$ with $||x -  (a_0, 0, i_0) || < \delta_0$:
\begin{align}
||\Phi_\ta(x) - \Phi_\ta  (a_0, 0, i_0)  ||_{[0, \tau]} < \frac{\delta_\infty}{2}.
\end{align}
This implies for all such trajectories $\Phi_\ta$:
\begin{align}
& 1) \text{ }  ||\Phi_\tau(x) - \Phi_{+\infty}(a_0,0,i_0) || < \delta_\infty, \text{ and} \\
&  2) \text{ } i(\ta) \geq \oc \qquad \forall \ta \in [0,\tau].
\intertext{In particular, $\Phi_\tau(x)$ lies within the $\delta_\infty$-neighborhood. We conclude that for all initial data $x \in \mathds{R}^3, ||x -  (a_0, 0, i_0) || < \delta_0$:}
& 3) \text{ } i(\ta) \geq \oc \qquad \forall \ta \geq 0.
\end{align}
Again, Theorem \ref{Triangles_3D_Thm} implies $ a(\ta),i(\ta)  \in T_c(\oc)$ for all $\ta \geq 0$. As before, the system is integrable and converges as $\ta \rightarrow + \infty$, so we can explicitly calculate $\ipinf(a_0,i_0)$. Thus, our claim holds for all starting points $(a_1,0,i_0)$, where $a_1$ lies in a small open interval around $a_0$.\\

\textbf{Step 3: limits of trajectories}\\
Assume that the claim holds for all $a_0 \in [0,a_u)$. For all trajectories starting in $(a_0,0,i_0)$, where $a_0 \in [0, a_u)$, it holds that $i(\ta)$ is monotone and bounded from below by $\oc$, such that $a(\ta),b(\ta) $ stay within $T_c(\oc)$. These trajectories are uniformly continuous in initial data on finite time-intervals. Fix any finite time-horizon $[0,T]$. The trajectories $\Phi_\ta  (a_0, 0, i_0) $ form a Cauchy-sequence on $||.||_{[0,T]}$ as $a_0 \rightarrow a_u$. Since $T$ can be chosen arbitrarily large and since all trajectories converge towards $\Phi_{+\infty}(a_0,0,i_0)$, which continuous in $a_0$, the claim also holds for the limiting trajectory that starts in $a_u$.\\

\textbf{Step 4: conclusion}\\
By step 1, the claim holds for $a_0$ in some small interval $[0, a_u)$. By step 3, it then also holds for $a_0=a_u$. If now $a_u<a^\ast$, the claim holds for $a_0 \in [0, a_u + \epsilon)$ by step 2 for some $\epsilon >0$. Iterating these two steps, the claim ultimately holds for all $a_0 \in \big [ 0, a^\ast(i_0) \big ] $. In particular, we have proven that the trajectories $\Phi_{z}(a_0,0,i_0)$ are uniformly continuous with respect to initial data on $z \in [0, + \infty]$. This continuity allows us to finish the proof of Proposition \ref{Prop:Tail_exp_rate_integrable}. In the non-critical cases where $c^2/4 + \ipinf -1 > 0$, the trajectories converge along a stable manifold with rate of convergence $-c/2 + \sqrt{c^2/4 + \ipinf -1}$. As $c^2/4 + \ipinf -1 \rightarrow 0$, the critical trajectory must converge along the limit of these manifolds.
\end{proof}
\end{theorem}

\section{The complete trajectory} \label{ch:Connection}

We follow the unstable manifold of points in $S_{-\infty}$, see \eqref{eq_unstable_set}, and show that it stays positive and enters the attractor of $S_{+\infty}$ from the previous section, cf. Theorem \ref{a_0_I_limit}.\\

\textbf{Assumption}: We will use the following setup over the entire Section \ref{ch:Connection}: Let $\iminf > 1$ and let $a(z), b(z), i(z)$ be the unique solution of the ODE-System \eqref{3D_Flow_Def} that emerges from $(0,0,\iminf)$ as $z \rightarrow - \infty$, such that $a(\ta) > 0$ asymptotically as $\ta \rightarrow - \infty$.\\

For all $\iminf > 1$, existence and uniqueness of these trajectories have been proven in Section \ref{ch:fix_point_results}. Moreover, we know their asymptotic behavior:

\begin{lemma} \label{Lem:neg_inf_tail}
The following holds as $\ta \rightarrow - \infty$:
\begin{equation}
\begin{aligned}
a(\ta) &>0,\\ b(\ta) &>0, \\ (a+i)' &< 0.
\end{aligned}
\end{equation}
\begin{proof}
The first two inequalities are given by Theorem \ref{Source_Theorem}, which also yields $b'(\ta) >0$ asymptotically. Noticing that $c(a+i)' = -a(1+r)-b' <0$ completes the proof.
\end{proof}
\end{lemma}

\subsection{The maximum of active particles}

For connecting these trajectories with $S_{+\infty}$, we need
\begin{proposition}[The maximum of active particles] \label{a0_reached}
There exists a finite phase-time $z_0$, such that $b(z_0) = 0$ for the first time.
\end{proposition}

Again, a certain monotonicity of the system comes at help. We will show that the sum $a(z) + i(z)$ decreases below $1$. Given this, the term $cb(z) + b'(z) = a(z) \cdot [ a(z) + i(z) -1]$ becomes negative, so $b$ must eventually reach $0$.

\begin{lemma} \label{lem:a_i_falling}
As long as $b(s)>0$ for all $s \in (-\infty,z]$, it holds that
\begin{align}
b(\ta) + i'(\ta) < 0.
\end{align}
\begin{proof}
In view of the asymptotic behavior of the trajectory, described in Lemma \ref{Lem:neg_inf_tail}, assume that there exists a finite time $\ta^\ast$, such that for the first time $ b(\ta^\ast) + i'(\ta^\ast)=0$, but still $b(\ta^\ast) > 0$. The wave-equations $ 0 = b'  + c b +a -a  (a+i)$ and $0 = c i' + a  (a+i) + r  a$ imply that
\begin{align}
0 = cb(\ta^\ast) +ci'(\ta^\ast)&  = -a(\ta^\ast) \cdot (1+r)-b'(\ta^\ast) \label{lem:i_infty_a+i'1} \\
& = c \cdot b(\ta^\ast) -a(\ta^\ast) \cdot \big[ a(\ta^\ast)+i(\ta^\ast)+r \big]. \label{lem:i_infty_a+i'2}
\end{align}
Rearranging the second line yields $cb(\ta^\ast) = a(\ta^\ast) \cdot \big[ a(\ta^\ast)+i(\ta^\ast)+r \big]$. As long as $\ta < \ta^\ast$, it holds that $b(\ta) + i'(\ta) < 0$, hence also
\begin{align}
cb(\ta) & < a(\ta) \cdot \big[ a(\ta)+i(\ta)+r\big].
\intertext{However, equality at $\ta = \ta^\ast$ implies that}
\frac{d}{d\ta} cb(\ta)\Big|_{\ta^\ast} & \geq \frac{d}{d\ta} a(\ta) \cdot \big[ a(\ta)+i(\ta)+r\big] \Big| _{\ta^\ast},
\end{align}
which we can rewrite, using both \eqref {lem:i_infty_a+i'1} and \eqref{lem:i_infty_a+i'2} :
\begin{align}
\begin{aligned}
c \cdot b'(\ta^\ast) &\geq b(\ta^\ast) \cdot \big [ a(\ta^\ast)+i(\ta^\ast)+r\big ] + a(\ta^\ast) \cdot \big [ b(\ta^\ast)+i'(\ta^\ast)\big ] \\ 
& = \frac{a(\ta^\ast)}{c} \cdot \big [ a(\ta^\ast)+i(\ta^\ast)+r \big]^2  + 0 \geq 0. \label{lem:i_infty_a+i'3}
\end{aligned}
\end{align}
But $a(\ta^\ast)>0$, so Eq. \eqref{lem:i_infty_a+i'1} implies that $b'(\ta^\ast) = -(1+r)a(\ta^\ast) <0$. This contradicts \eqref{lem:i_infty_a+i'3}.
\end{proof}
\end{lemma}

\begin{lemma}
As long as $b(s)>0$ for all $s \in (-\infty,z]$, it can not happen that $ a(\ta) + i(\ta) $ converges to some finite $L > 0$.
\begin{proof}
By the previous lemma: $(a+i)'<0$ while $b>0$. Assume that $ a(\ta) + i(\ta) $ converges to a finite value $L>0$ from above, which we denote as $ a(\ta) + i(\ta) \, \searrow \, L$. This implies that also $b+i' \, \nearrow \, 0$. By the Wave Equations \eqref{WAVE_EQ}, these two expressions are equivalent to
\begin{align}
-a(1+r)-b' & \, \nearrow \, 0 \quad \text{ and } \\ cb -a(a+i+r) &\, \nearrow \,  0.
\end{align}
The first convergence indicates that $b' \leq \delta <0$ after some time $\ta_\delta$, since $a$ is strictly increasing and hence positive. The second statement is equivalent to $cb -a \cdot (L+r)  \, \nearrow \, 0$. Thus, also $b$ is increasing. But $b'(\ta)<0$ for all $\ta \geq \ta_\delta$ and while $b>0$, a contradiction.
\end{proof}
\end{lemma}

We can now finish the
\begin{proof}[\textbf{Proof of Proposition \ref{a0_reached}}]
We now show that there exists a finite phase-time $\ta_0$ such that $b(\ta_0) = 0$ . By the previous lemma, $ a(\ta) + i(\ta) $ decreases below every positive value as long as $b(\ta)>0$. In particular, for some $\epsilon >0$: $a(\tau)+i(\tau) \leq 1-\epsilon$ after some phase-time $\tau$. Then for all $\ta \geq \tau$, since $a > 0$:
\begin{equation}
\begin{aligned}
cb(\ta)+b'(\ta) &= a(\ta) \cdot \big [ a(\ta) + i(\ta) -1 \big ]  \\ &\leq a(\ta) \cdot ( 1-\epsilon -1) = -a(\ta) \epsilon < 0.
\end{aligned}
\end{equation}
Either $cb(\ta) < 0$ and the system has already passed a first local maximum of $a(\ta)$, or we may assume that $b'(\ta) \leq -a(\tau) \epsilon = - \delta <0$. If now $b'(\ta) \leq -\delta$, then $b(\ta)$ ultimately reaches zero, say at some phase-time $\ta_0$. The bound $b'(\ta) \leq - \delta <0$ for all $\ta \geq \tau$ ensures that indeed $\ta_0<+\infty$.
\end{proof}

\subsection{Reaching the attractor of $\boldsymbol{S_{+ \infty}}$}
\label{ch:a_ast}
We now prove that $(a_{\ta_0},0,i_{\ta_0})$ lies in the attractor of the stable set $S_{+\infty}$. Therefore, we show that $a_{\ta_0} \leq a^\ast(i_{\ta_0})$, then Theorem \ref{a_0_I_limit} ensures non-negativity and convergence as $z \rightarrow + \infty$. We again use Proposition \ref{Prop_integrals}, now to interrelate $(0,0,\iminf)$ and $(a_{z_0},0,i_{z_0})$:

\begin{lemma} \label{I-_a0_Lemma}
The following holds at phase-time $\ta_0$:
\begin{align}
a_{\ta_0} &> 0 , \quad a_{\ta_0} + i_{\ta_0} \leq 1, \\
a_{\ta_0} &= \frac{c^2}{c^2+1+r} \Bigg\{ -(i_{\ta_0}+r) \nonumber \\ & \quad + \sqrt{
(i_{\ta_0}+r)^2 + \frac{c^2+1+r}{c^2} \Big( 
(\iminf-1)^2 - (1-i_{\ta_0})^2
\Big )
} \Bigg \}. \label{I_-inf_to_a0}
\end{align}
In the case $ \iminf \in (1,2 - \oc]$, then additionally
\begin{align}
i_{\ta_0}  \in (\oc,1), \qquad a_{\ta_0}  \in (0,1),
\end{align}
and the trajectory is non-negative for $z \in (-\infty, \ta_0]$.
\begin{proof}
As $\ta \rightarrow -\infty$, all $a(\ta), b(\ta),b'(\ta),i'(\ta)$ have exponential and hence integrable tails, cf. Theorem \ref{Source_Theorem}. We thus can apply Proposition \ref{Prop_integrals}. Solving the resulting system of equations results in Eq. \eqref{I_-inf_to_a0}, we omit the intermediate steps. It holds that $a_{\ta_0}>0$ because $b(\ta)>0$ for all $\ta < \ta_0$.

In particular, $a_{\ta_0} > 0$ implies that the second summand under the root in \eqref{I_-inf_to_a0} must be strictly positive, which yields $(\iminf-1)^2  > (1-i_{\ta_0})^2$. Since $b(\ta_0)=0$ for the first time, it must hold that $b'(\ta_0) \leq 0$. Given this, we use $b'(\ta_0) + cb(\ta_0) = a_{\ta_0}(a_{\ta_0}+i_{\ta_0}-1)$ to bound $0 \geq a_{\ta_0}(a_{\ta_0}+i_{\ta_0}-1)$. Since $a_{\ta_0}>0$, this shows that $i_{\ta_0} \leq 1-a_{\ta_0} < 1$.

If we assume additionally that $\iminf \in (1,2-\oc]$, then $(\iminf-1)^2  > (1-i_{\ta_0})^2$ implies that $i_{\ta_0} > 2 - \iminf \geq \oc $. Up to $\ta_0$, $a(z)+i(z)$ is decreasing, which was proven in Lemma \ref{lem:a_i_falling}. Since $a(z)$ is strictly increasing up to $\ta_0$, $i(z)$ is strictly decreasing, but not below $i(z_0) > 0$. Hence, the trajectory stays positive. The inequality $a_{z_0} + i_{z_0}\leq 1$ implies that $a_{z_0} < 1$. 
\end{proof}
\end{lemma}

Finally, we connect the unstable manifold of $(0,0,\iminf)$ with the attractor of $S_{+\infty}$:
\begin{proposition}[Reaching the attractor of $S_{+ \infty}$] \label{GOOD_leaving}
Let $\iminf \in (1,2 - \oc]$. The non-negative branch of the unstable manifold of $(0,0,\iminf)$ reaches the point $(a_{\ta_0},0,i_{\ta_0})$, where $a_{\ta_0} \in (0,1)$ and $i_{\ta_0} \in (\oc,1)$. It then holds that
\begin{align}
0<a_{\ta_0} \leq a^\ast(i_{\ta_0}),
\end{align}
for $a^\ast$ like in Definition \ref{A_ast_def}. In view of Theorem \ref{a_0_I_limit}, the trajectory that starts/continues in such a point $(a_{\ta_0}, 0, i_{\ta_0})$ converges to $S_{+\infty}$ as $\ta \rightarrow + \infty$ and stays non-negative.

\begin{proof}
We have just shown that $(i_{\ta_0}, a_{\ta_0}) \in (0,1)^2$ and that $a_{\ta_0} + i_{\ta_0} \leq 1$. Recall Definition \ref{A_ast_def}: $a^\ast(i_0) = \min \{ \alpha(i_0), 1- i_0 \}$, where $\alpha(i_0)$ is given by
\begin{align}
\begin{aligned}
	\alpha(i_0) &= \frac{c^2}{1+c^2+r}
	\Bigg\{ -(i_0 +r) \\ & \quad + \sqrt{ (i_0 +r)^2 + \frac{c^2+1+r}{c^2}
	\Big ( (1-\oc)^2 - (1-i_0)^2 \Big )	} \Bigg\}. \label{eq:expr_1}
	\end{aligned}
	\end{align}
We have already verified that $a_{\ta_0} + i_{\ta_0} \leq 1$, so proving $a_{\ta_0} \leq \alpha(i_{\ta_0})$ suffices for proving $a_{\ta_0} \leq a^\ast(i_{\ta_0})$. By \eqref{I_-inf_to_a0}, we know that
\begin{align}
\begin{aligned}
\hspace{1cm} a_{\ta_0} &= \frac{c^2}{c^2+1+r} \Bigg\{ -(i_{\ta_0}+r) \\ & \quad + \sqrt{
(i_{\ta_0}+r)^2 + \frac{c^2+1+r}{c^2} \Big( 
(\iminf-1)^2 - (1-i_{\ta_0})^2
\Big)
} \Bigg\}. \label{eq:expr_2}
\end{aligned}
\end{align}
The two expressions \eqref{eq:expr_1} and \eqref{eq:expr_2} are very similar. After some elementary steps, the claim $a_{\ta_0} \leq \alpha(i_{\ta_0})$ is equivalent to
\begin{align}
(\iminf -1)^2 \leq (1- \oc)^2.
\end{align}
This is equivalent to $\iminf \leq 2- \oc$, since $\iminf >1$ and $ \oc \leq 1$. But that is just how we have chosen $\iminf$.
\end{proof}
\end{proposition}

\section{Concluding the proof of the main result}
\label{ch:main_thm_proof}

We bring together our results from the previous sections and complete the

\begin{proof}[\textbf{Proof of Theorem \ref{Main_Theorem}}]
Let $\iminf \in (1,2 - \oc]$. We consider the ODE System \eqref{3D_Flow_Def} in coordinates $a,b,i$. The unstable manifold of the fixed point $(a,b,i)=(0,0,\iminf)$ has dimension one. Its two branches are the only trajectories that leave the fixed point, which is stated in Theorem \ref{Source_Theorem}. There is one branch of the unstable manifold such that $a(\ta)>0$ as $\ta \rightarrow -\infty$, we follow this trajectory in positive direction of $\ta$. There is a finite phase-time $\ta_0$, such that for the first time $b(\ta_0) = 0$, see Proposition \ref{a0_reached}. Up to time $\ta_0$, $b(\ta)>0$ and $a(\ta),i(\ta)> 0$. Denote the state of the system at $\ta_0$ as $(a_{\ta_0},0,i_{\ta_0})$. Lemma \ref{I-_a0_Lemma} states that $i_{\ta_0} \in (\oc,1)$, Proposition \ref{GOOD_leaving} states that $a_{\ta_0} \in (0, a^\ast(i_{\ta_0})]$, for $a^\ast$ as in Definition \ref{A_ast_def}. By Theorem \ref{a_0_I_limit}, we then know that $(a_{\ta_0},0,i_{\ta_0})$ lies in a non-negative attractor of the set $S_{+\infty}$. Thus, $a(\ta),b(\ta),i(\ta) \rightarrow (0,0,\ipinf)$ as $\ta \rightarrow + \infty$, where $\ipinf \in [\oc,1)$, and ultimately $a(\ta),i(\ta) \geq 0$ for all $\ta \in \mathds{R}$.

For any non-negative and bounded solution, the identity $\iminf + \ipinf = 2$ holds by Proposition \ref{nec_cond_thm}. For $c>0$ and $\iminf \in (1,2 - \oc]$, the previous paragraph proves existence and uniqueness of the claimed wave. For $\iminf = 1$, the constant solution can be the only non-negative and bounded one.

Now assume that there exists a non-constant, bounded and non-negative solution. By monotonicity of $i(z)$, it must converge as $z \rightarrow \pm \infty$. If $\iminf \in (1,2-\oc]$, it is one of the above solutions. If $\iminf > 2-\oc$, then $\ipinf < \oc$. In this case, the trajectory can not stay non-negative as $z \rightarrow + \infty$, which is stated by Proposition \ref{Nec_oscillation_convergence}.
\end{proof}

\section{Discussion and outlook at stability} \label{ch:discussion}
\subsection{FKPP-waves}

We have given a description of all bounded and non-negative traveling waves of the Reaction-Diffusion System \eqref{EQUA}. For the most related systems, the FKPP-equation \cite{KPP_1937_Wave, Fisher_1937_Wave}, the FitzHugh-Nagumo-equation \cite{FitzHugh_equation, Nagumo_equation} and combustion equations \cite{Berestycki_Scheurer_Trav_Combustion}, no such continuum of traveling waves has yet been constructed.

Apart from that, the traveling waves of System \eqref{EQUA} are closely related to pulled FKPP-waves with only a single type of particles \cite{Fisher_1937_Wave, KPP_1937_Wave}. The equation for such a wave $w(z)$ reads $0 = c w' + w'' + F(w)$. For the purpose of a simple comparison, we let $F(w) = aw - w^2$, where $a>0$ is the branching rate of the particles. In this case, Theorem \ref{Main_Theorem} states that the convergence of System \eqref{EQUA} as $z \rightarrow + \infty$ is identical to that of $w$, if $ a = 1-  \ipinf$, see e.g. \cite{Uchiyama1977}. In words, the asymptotic growth speed of traveling waves of System \eqref{EQUA} coincides with that of simple FKPP-waves in presence of a constant density $\ipinf$ of inhibiting particles. Moreover, Theorem \ref{Main_Theorem} implies that $\oc = 0$ for all $c \geq 2$. Thus, the minimal speed of an invasive front, where $\ipinf = 0$, is given by $c_{\min}=2$. Again, this coincides with the minimal wave speed of the associated FKPP-equation, \textit{i.e.} in the absence of inactive particles. It is this critical front which can be interpreted as the most natural one, our simulations indicate that it always arises under compact initial data. If we assume convergence, a technique of Berestycki, Brunet \& Derrida \cite{Berestycki_2018_Wave_Front} yields an upper bound for the speed of the traveling front, just by ignoring the dampening influence of the inactive particles. For compact initial data, the system always chooses the smallest possible wave speed, as suggested.

The emergence of traveling fronts is known for many reaction-diffusion systems. We suggest the
literature \cite{Britton_BiologyEssentials, Volpert_waves_biology, Othmer2009WavesinBiology} for more examples with a biological motivation. Rigorous proofs of these phenomena are rare. Typically, only the form of the traveling waves is analyzed analytically. The FKPP-equation is one of the cases, where the convergence of the front of the PDE towards a traveling wave solution can be proved. The first rigorous proof was done by Kolmogorov, Petrovsky \& Piscunov in 1937 \cite{KPP_1937_Wave}. Extensions of this result to more general initial data and a more precise description of the speed of the front have been provided by Uchiyama \cite{Uchiyama1977} and M. Bramson \cite{Bramson_1983ConvergenceOS}. The approach of Kolmogorov \textit{et al.} and Uchiyama seems to be restricted to systems with only a single type of particles, as it relies on a maximum principle and monotonicity of the front. The approach of Bramson relies on a relationship between the FKPP-equation and branched Brownian motion, which can not be applied in the present case since the inactive particles do not diffuse. A singular perturbation of System \eqref{EQUA} which introduces a small diffusion to the inactive particles will be subject to future investigations. This would also rule out some difficulties when analyzing the stability of the traveling waves against perturbations, discussed in the next section.

\subsection{Stability of the traveling waves}

We give a brief introduction to stability of traveling waves against small perturbations, in the spirit of the introduction in \cite{Ghazaryan_overview}. A good overview, where the following concepts are presented in depth, has been written by Sandstede \cite{SANDSTEDE_stability_traveling}.

Consider a reaction-diffusion system
\begin{align}
Y_t &= D \cdot Y_{xx} + R(y), \label{Eq:Reaction_System_Disc}
\intertext{where $Y \in \mathds{R}^n, x \in \mathds{R}, t \geq 0, D = \text{diag} (d_1, \dots , d_n)$ with $d_i \geq 0$, and $R$ a smooth reaction. In the moving frame $z = x-ct$, the System reads }
Y_t &= D \cdot Y_{zz} + c Y_z + R(y). \label{Eq:Shifted_System_Disc}
\end{align}
A traveling wave $w(z)$ with speed $c$ is a constant solution of Eq. \eqref{Eq:Shifted_System_Disc}. The wave $w$ is called \textit{non-linearly stable} in a space $\mathcal{X}$, if any solution of the PDE \eqref{Eq:Shifted_System_Disc} which starts in $Y_0 = w + \tilde{Y}$, where $\tilde{Y} \in \mathcal{X}$ is a sufficiently small perturbation, converges to a shift of $w$. This type of stability is often encoded in the spectrum of the operator $\mathcal{L}$, that is obtained by linearizing the equation for the perturbation $\tilde{Y}$ in \eqref{Eq:Shifted_System_Disc} around to the constant part $w$:
\begin{align}
\tilde{Y}_t = D \cdot \tilde{Y}_{zz} + c \tilde{Y}_z + JR(w) \cdot \tilde{Y} := L \tilde{Y}, \label{Eq:Linearized_Operator}
\end{align}
where $JR$ is the Jacobian of the reaction $R$. Let $\mathcal{L}:\mathcal{X} \rightarrow \mathcal{X}$ be the operator given by $\tilde{Y} \rightarrow L \tilde{Y}$. We say that the wave $w$ is \textit{spectrally stable} in $\mathcal{X}$ if the spectrum of $\mathcal{L}$ is contained in the half-plane $\mathfrak{Re}(\gamma) < 0$, except maybe a simple a simple eigenvalue at $0$ (that corresponds to the traveling wave itself, if $w' \in \mathcal{X}$). For diffusive systems, a quite general theory has been developed. If $\mathcal{X}$ is appropriately chosen, spectral stability implies non-linear stability, we refer to the literature \cite{SANDSTEDE_stability_traveling, Ghazaryan_overview}. Classical results are e.g. given for subspaces of $\mathcal{X} = H^1$, the $L^2$-Sobolev space.

To cut a long story short, we are not aware of any rigorous framework for studying the non-linear stability of System \eqref{EQUA}. Two problems arise, that so far have been treated only separately \cite{Ghazaryan_overview, Kirchgaessner_critical_fronts}.

Most importantly, the traveling waves of System \eqref{EQUA} can not be stable against perturbations in the classical sense, since the inactive particles neither react nor diffuse. Any initial deviation remains for all times, as shown in Figure \ref{Diff_system_pics}. However, the actual front of the system does converge to a traveling wave. For capturing this idea, we introduce the weighted space $\mathcal{X} = H^1_\alpha$ with norm $||f||_{H^1_\alpha} = ||f \cdot e^{\alpha z}||_{H^1}$ for some $\alpha >0$. Non-linear stability in $H_1^\alpha$ is referred to as \textit{convective stability}. Convergence of the PDE in the moving frame \eqref{Eq:Shifted_System_Disc} in $H^1_\alpha$ means that the front of the system approaches the traveling wave, whereas any initial finite and local deviation is convected towards $z= -\infty$ and vanishes due to the weighting. A first rigorous result regarding convective stability was obtained by Ghazaryan \textit{et al.} \cite{Ghazaryan_overview}. They could show that spectral stability in $H^1_\alpha$ implies convective stability against small perturbations in $H^1_\alpha \cap H^1$. For their approach, the authors require that the weight $\alpha$ can be chosen such that all eigenvalues $\gamma$ of $\mathcal{L}$ except zero fulfill $\mathfrak{Re}(\gamma) \leq \nu < 0$ and such that the derivative $w' \in H^1_\alpha$ of the traveling wave is an eigenfunction that corresponds to a simple eigenvalue at zero. Unfortunately, this setting is not suited for studying pulled FKPP-fronts: the assumption $w' \in H^1_\alpha$ implies that the continuous spectrum of $\mathcal{L}$ touches the origin, see e.g. chapter 6 in the work of Sattinger \cite{SATTINGER_stability}.

Another difficulty arises when studying critical pulled fronts (with minimal possible speed) whose tail as $z \rightarrow + \infty$ converges sub-exponentially, as in Theorem \ref{Main_Theorem}. In this case, the requirement $w' \in H^1_\alpha$ is only fulfilled for rather small values of $\alpha$, which do not suffice for shifting the continuous spectrum of $\mathcal{L}$ to the left half-plane. For diffusive systems, this more delicate case was first treated by Kirchgässner \cite{Kirchgaessner_critical_fronts}, a recent overview is given in \cite{Faye2018_critical_FKPP}. In contrast to non-critical waves, the type of convergence of the system to the critical traveling wave is not exponential, but algebraic. For the convergence of the FKPP-equation with compact initial data to its critical traveling wave, this is known since the pioneering work of Kolmogorov \textit{et al.} \cite{KPP_1937_Wave}. After introducing a small diffusion to the inactive particles, we could use the result of Kirchgässner, this seems like the next natural step.

For the most natural traveling wave solution, the critical one with speed $c=2$ and $\ipinf = \oc =  0$, we performed a numerical analysis that strongly indicates that this wave is spectrally stable in $H^1_\alpha$, when we choose $\alpha = - \mu_{ + \infty}=c/2$. The details are presented in Appendix \ref{ch:spectrum_evans}. Thus, based on the work of Ghazaryan \textit{et al.} regarding convective stability \cite{Ghazaryan_overview} and the work of Kirchgässner regarding critical fronts \cite{Kirchgaessner_critical_fronts}, we dare to make an educated guess: we expect that this traveling wave is convectively stable against small perturbations in $H^1_{c/2} \cap H^1$, with algebraic speed of convergence.\\

\textbf{Acknowledgment:} The author would like to thank Anton Bovier and Muhittin Mungan for their support and the fruitful discussions. He also would like to thank the anonymous referees who provided very useful and detailed comments on a previous version of the manuscript.

\textbf{Numerical analysis:} The spectral analysis was performed numerically via STABLAB \cite{Barker_Stablab}, which is a MATLAB-library exactly for this purpose. The simulations of the Reaction-Diffusion System \eqref{EQUA} were performed via Wolfram Mathematica. The code can be accessed upon request.

\bibliography{Bibliothek.bib}{}
\bibliographystyle{plain}

\appendix
\section*{Appendix}

\section{Center manifold calculations} \label{ch:Center_manifold}
\subsection{Review of center manifold theory}
\label{ch:center_manifold_theory}
\begin{definition}[Normal form] \label{Def_Normal_Form}
The normal form of a dynamical system $dx/dt = f(x), x \in \mathds{R}^n$ around its fixed point $0 \in \mathds{R}^n$ is defined as follows. Write $x=(y,z)$ where $y \in \mathds{R}^k, z \in \mathds{R}^l$ and $k+l = n$, such that $dx/dt = f(x)$ is equivalent to:
\begin{equation}
\begin{aligned}
\frac{dy}{dt} &= A \cdot y + g(y,z),  \\
\frac{dz}{dt} &= B \cdot z + h(y,z). \label{Normal_dynamics}
\end{aligned}
\end{equation}
We require that the eigenvalues of $A \in \mathds{R}^{k \times k} $ have zero real parts and those of $B \in \mathds{R}^{l \times l}$ have nonzero real parts. Further, both functions $g: \mathds{R}^n \rightarrow \mathds{R}^k$ and $h: \mathds{R}^n \rightarrow \mathds{R}^l$ are smooth and vanish together with their first-order partial derivatives at the origin.
\end{definition}

\begin{proposition} \label{Existence_Normal_Form}
Let $f: \mathds{R}^n \rightarrow \mathds{R}^n$ be smooth. Let a dynamical system $dx / dt = f(x), x \in \mathds{R}^n$ have a fixed point $x_0 \in \mathds{R}^n$, such that the eigenvectors of the Jacobian $Df(x_0)$ span the entire $\mathds{R}^n$. The system can be written in normal form as in Definition \ref{Def_Normal_Form}.
\end{proposition}
The proof includes a simple but technical change of coordinates into the system of eigenvectors of the Jacobian $Df(x_0)$. This will be done explicitly in Section \ref{ch:CM_calculations}. For the underlying theory and the more general case, we refer to the monograph of U. Kirchgraber \& K.J. Palmer \cite{Kirchgraber_Palmer_Flows_Invar_Manifolds}.

\begin{definition}[Center manifold]
Consider a dynamical system in normal form \eqref{Normal_dynamics}. Let $\phi: \mathds{R}^k \rightarrow \mathds{R}^l$ be a smooth function such that $\phi(0) = 0$ and also its derivative $D \phi(0) = 0$. Assume that the set
\begin{align}
\mathcal{CM} = \Big \{ y \in \mathds{R}^k, z \in \mathds{R}^l : \, z = \phi(y) \Big \} \label{eq:cm_notation}
\end{align}
is invariant under Dynamics \eqref{Normal_dynamics}. It is then called a \textit{center manifold} of the fixed point (due to its vanishing derivative at $0$).
\end{definition}

We will use a local version of the center manifold, which can be shown to exist in a neighborhood of the fixed point:

\begin{theorem}[Local center manifold, cf. Theorem 4.1. in \cite{Kirchgraber_Palmer_Flows_Invar_Manifolds}]  \label{CM_Flow_Thm}
Consider a smooth dynamical system in normal-form \eqref{Normal_dynamics}, where $\dim(y) =k \geq 1$, such that the Jacobian at the fixed point has $k$ eigenvalues with zero real part. Let $c_1 + c_2 = \dim(z)$, where the matrix $B$ has $c_1$ eigenvalues with positive real part and $c_2$ eigenvalues with negative real part. Then locally, there exist a unique center manifold of dimension $k$, a unique unstable manifold of dimension $c_1$ and a unique stable manifold of dimension $c_2$.

The center manifold can be written as $\big\{ (y,z) : \, z = \phi(y) \big\}$ like in \eqref{eq:cm_notation}. There exists a homeomorphism defined in an open neighborhood of the origin which takes solutions of $dx / dt = f(x)$ onto solutions of
\begin{equation}
\begin{aligned}
\frac{dy}{dt} &= A \cdot y + g\big( y,\phi(y) \big) ,  \\
\frac{dz}{dt} &= B \cdot z. \label{eq:normal_form_hom}
\end{aligned}
\end{equation}
\end{theorem}

\begin{definition}[Error of approximation of the center manifold]
Consider a smooth dynamical system in normal-form \eqref{Normal_dynamics}. For a smooth function $T: \mathds{R}^k \rightarrow \mathds{R}^l$ define the error of approximation of the normal form by
\begin{equation}
\begin{aligned}
(HT)(y) & := DT(y) \cdot \big[ Ay + g \big( y,T(y) \big) \big] \\ &  \, - B \cdot T(y) - h(y,T(y)) .
\end{aligned}
\end{equation}
\end{definition}

\begin{theorem}[Approximating the center manifold, cf. Theorem 3 in \cite{Carr_Center_Manifold}]  \label{Approx_Center_MF}
Consider a smooth dynamical system in normal form \eqref{Normal_dynamics} with local center manifold $\{ (y,z) : \, z = \phi(y) \}$ as in \eqref{eq:cm_notation}. Let $T: \mathds{R}^k \rightarrow \mathds{R}^l$ be smooth with $T(0) = 0$ and $DT(0) = 0$. Suppose that as $y \rightarrow 0$, for some $q >1$:
\begin{align}
(HT)(y) &= \mathcal{O}(|y|^q).
\intertext{Then, as $y \rightarrow 0$, also}
|T(y) - \phi(y)| &= \mathcal{O}(|y|^q).
\end{align}
\end{theorem}

\subsection{Calculating the normal form and the center manifold} \label{ch:CM_calculations}
We analyze the flow of the ODE System \eqref{3D_Flow_Def} around its fixed points by applying the theory from the previous section. We therefore write the system into normal form, see Def. \ref{Def_Normal_Form}. For a fixed point $(a,b,i) = (0,0,K)$, we begin with the affine transformation
\begin{align}
j = i-K,
\end{align}
and then decompose the resulting system into a linear part $M$ and a non-linear part $G$. To be concise with the notation from the previous section, which is adopted from the existing literature, we use a vectorial notation in coordinates $(j,a,b)$, such that the center manifold can be written as $\big \{ (j,a,b) : \, (a,b) = \phi(j) \big \}$.
\begin{definition} \label{Def:App_lin_nonlin}
Given $c >0, K \in \mathds{R}$, introduce the matrix $M$ as
\begin{align}
M &:= \begin{pmatrix}
0 & -\frac{K-r}{c} & 0 \\
0 & 0 & 1  \\
0 & K-1 & -c
\end{pmatrix}.
\intertext{Further, define the non-linear functions $g(j,a) := a^2 + aj$ and $G: \mathds{R}^3 \rightarrow \mathds{R}^3$:}
G
\begin{pmatrix}
j \\
a \\
\apr
\end{pmatrix} &  :=  g(j,a) \cdot \begin{pmatrix}
- \frac{1}{c} \\
0 \\
1
\end{pmatrix}. \label{Def:Gnonlinear}
\end{align}
\end{definition}

\begin{lemma}[Linear and non-linear part]
For $c>0, K \in \mathds{R}$, the ODE System \eqref{3D_Flow_Def} can be decomposed in its linear and non-linear part. In coordinates $(j,a,\apr)$, where $j = i-K$, and using Def. \ref{Def:App_lin_nonlin}, this reads as
\begin{align}
\begin{pmatrix}
j' \\ a' \\ \apr'
\end{pmatrix} = 
M \cdot \begin{pmatrix}
j \\
a \\
\apr
\end{pmatrix} + 
G
\begin{pmatrix}
j \\
a \\
\apr
\end{pmatrix}.
\end{align}
\end{lemma}

\begin{definition}
For given $c>0,K \in \mathds{R}$, we define the discriminant
\begin{align}
\Delta := \sqrt{ \frac{c^2}{4} + K -1}.
\end{align}
The eigenvalues and eigenvectors of $M$ are then given by (cf. \eqref{Eigenvalues_DF_intro})
\begin{align}
&\lambda_0 = 0, \hspace{0.28cm} \lambda_{\pm} = -\frac{c}{2} \pm \Delta, \\
e_0 & = \begin{pmatrix}
1 \\ 0 \\ 0
\end{pmatrix}, 
e_\pm  = 
\begin{pmatrix}
\frac{K+r}{c} \cdot \frac{\lambda_\mp}{\lambda_\pm} \\
  -\lambda_\mp \\
  K-1 \end
  {pmatrix}. \label{Eigenvalues_Normal}
\end{align}
\end{definition}

Technical difficulties arise when the eigenvectors no longer span the entire $\mathds{R}^3$. We require $K\neq 1$ in view of \eqref{Eigenvalues_Normal}, which also eliminates the case $\lambda_+ = 0$. For similar reasons, we also exclude the case that $\lambda_+ = \lambda_-$, so we require that $\Delta \neq 0$. This is given if $K \neq 1-c^2/4$.

\begin{lemma}
Let $c>0$ and $K \notin \{1,1-c^2/4 \}$. The matrix $M$ can be written in diagonal form, such that $M = E D E^{-1}$. The matrices $D, E, E^{-1}$ are given by:
\begin{align}
D &= \textup{diag}(\lambda_0, \lambda_+, \lambda_-), \\
E &=
	\begin{pmatrix}
\text{ } & \text{ } & \text{ } \\
e_0 & e_+ & e_- \\
\text{ } & \text{ } & \text{ } \\
	\end{pmatrix} =
	\begin{pmatrix}
1 & \frac{K+r}{c} \cdot \frac{\lambda_-}{\lambda_+} & \frac{K+r}{c} \cdot \frac{\lambda_+}{\lambda_-}  \\
0 & -\lambda_- & -\lambda_+ \\
0 & K-1 & K-1
\end{pmatrix}, \\
E^{-1} &= 
	\begin{pmatrix}
1 & -\frac{K+r}{(1-K)} & -\frac{K+r}{c(1-K)} \\
0 & \frac{1}{2 \Delta } & - \frac{\lambda_+}{2 \Delta (1-K)} \\
0 & -\frac{1}{2 \Delta } & \frac{ \lambda_-}{2 \Delta (1-K)}
\end{pmatrix}.
\end{align}
\end{lemma}

\begin{lemma}[Dynamics in normal form] \label{NORMAL_FORM}
Let $c>0$ and $K \notin \{1,1-c^2/4 \}$. The eigenvectors $e_0, e_+, e_-$ of $M$ form a basis of $\mathds{R}^3$. We introduce the coordinates $(u,v,w)$, such that any $x  \in \mathds{R}^3$ can be written as $x = u \cdot e_0 + v \cdot e_+ + w \cdot e_-$. The System \eqref{3D_Flow_Def} in coordinates $(u,v,w)$ follows dynamics given via
\begin{align}
\begin{pmatrix}
u' \\ v' \\ w'
	\end{pmatrix}
	= \begin{pmatrix}
0 \\ \lambda_+ \, v \\ \lambda_- \, w 
\end{pmatrix} 
+ P(u,v,w) \cdot \begin{pmatrix}
-\frac{1}{c} ( 1 + \frac{K+r}{1-K}) \\
- \frac{\lambda_+}{2 \Delta (1-K)} \\
\frac{\lambda_-}{2 \Delta (1-K)}
	\end{pmatrix},
\end{align}
where $P$ is a polynomial such that $P(u,0,0) = 0$:
\begin{align}
P(u,v,w) &:= -\Big( \lambda_- \, v + \lambda_+ \, w \Big) \nonumber \\
& \hspace{0.65cm} \cdot \Big(
-\lambda_- \, v - \lambda_+ \, w + u + \frac{K+r}{c} \big[ \frac{\lambda_-}{\lambda_+} v + \frac{\lambda_+}{\lambda_-} w \big]
\Big). \label{def:pol_cm}
\end{align}
\begin{proof}
We change coordinates from $u,v,w$ to $j,a,b$ and back:
\begin{align}
\begin{pmatrix}
u' \\ v' \\ w'
\end{pmatrix}
& = \begin{pmatrix}
0 \\ \lambda_+ \, v \\ \lambda_- \, w 
\end{pmatrix} 
+ E^{-1} \cdot G ( E \cdot \begin{pmatrix}
u \\ v \\ w
\end{pmatrix} ).
\intertext{Now recall the functions $G$ and $g$, see \eqref{Def:Gnonlinear}. Explicitely calculating the non-linear part $E^{-1}GE$ results in}
E^{-1} \cdot G ( E \cdot \begin{pmatrix}
u \\ v \\ w
\end{pmatrix} )
&= E^{-1} \cdot
	\begin{pmatrix}
- \frac{1}{c} \\ 0 \\ 1
\end{pmatrix} \cdot g(E \cdot \begin{pmatrix}
u \\ v \\ w
	\end{pmatrix})
	\nonumber \\
&= \begin{pmatrix}
-\frac{1}{c} ( 1 + \frac{K+r}{1-K}) \\
- \frac{\lambda_+}{2\Delta(1-K)} \\
\frac{\lambda_-}{2\Delta(1-K)}
	\end{pmatrix}
	\cdot 
	 g \begin{pmatrix}
u + \frac{K+r}{c} ( \frac{\lambda_-}{\lambda_+} v + \frac{\lambda_+}{\lambda_-} w)  \\
- \lambda_- \, v - \lambda_+ \, w \\
\apr \big( u,v,w \big)
\end{pmatrix} \nonumber \\
& = P(u,v,w) \cdot \begin{pmatrix}
-\frac{1}{c} ( 1 + \frac{K+r}{1-K}) \\
- \frac{\lambda_+}{2\Delta(1-K)} \\
\frac{\lambda_-}{2\Delta(1-K)}
	\end{pmatrix}.
\end{align}
Luckily, for evaluating $g(j,a)$, we do not have to calculate the coordinate $b(u,v,w)$. 
\end{proof}
\end{lemma}

Now we have all ingredients for computing the center manifold. We use the approximation argument from Theorem \ref{Approx_Center_MF}.

\begin{theorem}[Asymptotic behavior] \label{CM_thm_main}
Let $c>0$ and $K \notin \{1,1-c^2/4 \}$, and let $(a,b,i) = (0,0,K)$ be a fixed point of the System \eqref{3D_Flow_Def}. Locally around $(0,0,K)$, the center manifold of the fixed point coincides with the set
\begin{align}
\{a=\apr=0\}.
\end{align}
In a non-empty open neighborhood around $(0,0,K)$, the flow of the System \eqref{3D_Flow_Def} is equivalent to
\begin{align}
\begin{pmatrix}
u' \\ v' \\ w'
	\end{pmatrix}
	= \begin{pmatrix}
0 \\ \lambda_+ \, v \\ \lambda_- \, w
\end{pmatrix}, \label{normal_flow_calculated}
\end{align}
where $u,v,w$ are the coordinates in the system of eigenvectors $e_0, e_+, e_-$ of the matrix $M$, see \eqref{Eigenvalues_Normal}.
\begin{proof}
In the normal form from Lemma \ref{NORMAL_FORM}, the center manifold can be calculated as a function $\phi(u): \mathds{R} \rightarrow \mathds{R}^2$. As $u \rightarrow 0, \phi(u) \in \mathcal{O}(u^2)$, and $\phi(u)$ can be approximated to any degree by some polynomial without linear and constant parts. For some arbitrary approximation $T: \mathds{R} \rightarrow \mathds{R}^2$ with components $T_v, T_w$, we can estimate the error of the approximation $(HT)(u)$ by Theorem \ref{Approx_Center_MF}. Inserting the normal form from Lemma \ref{NORMAL_FORM} results in
\begin{equation}
\begin{aligned}
(HT)(u) = &- DT(u) \cdot P\big(u,T(u)\big) \cdot \frac{1}{c}(1+\frac{K-r}{1-K})
- \begin{pmatrix}
\lambda_+ \cdot T_v(u) \\
\lambda_- \cdot T_w(u)
\end{pmatrix}
 \\
 &- P\big(u,T(u)\big) \cdot \begin{pmatrix}
- \frac{\lambda_+}{2\Delta(1-K)} \\
\frac{\lambda_-}{2\Delta(1-K)}
	\end{pmatrix}.  \label{Approx_Center}
\end{aligned}
\end{equation}
For the center manifold, $(H\phi)(u)=0$. From Eq. \eqref{Approx_Center}, we can extract the coefficients of the Taylor Expansion of $\phi(u)$ around the fixed point iteratively, by choosing better and better approximating polynomials $T_n$. For the start, take some polynomial $T_2(u): \mathds{R} \rightarrow \mathds{R}^2$ of order $2$. Let $\alpha, \beta \in \mathds{R}$ and define
\begin{align}
T_2(u) &:= (\alpha u^2, \beta u^2).
\intertext{Note that $P\big(u,T_2(u)\big) = \mathcal{O}(u^3)$, for $P$ as defined in \eqref{def:pol_cm}. Thus}
(HT_2)(u) &= \mathcal{O}(u^3) - \begin{pmatrix}
\lambda_+ \, \alpha u^2 \\
\lambda_- \, \beta u^2
\end{pmatrix}.
\end{align}
Hence, for any approximation of type $T_2(u) = (\alpha u^2, \beta u^2)$, the leading error term is of order $\mathcal{O}(u^3)$ if and only if $T_2(u) \equiv (0,0)$. We conclude that the second order approximation of $\phi$ is given by $T_2(u) \equiv (0,0)$. By an easy induction, it follows that $T_n(u) = (0,0)$ for all $n \geq 2$, and so the local center manifold is given by $\phi(u) = (0,0)$. In the original system, this corresponds to $\{a=b=0\}$, which are the fixed points of the ODE \eqref{3D_Flow_Def}. We can now calculate the asymptotic flow in the normal form, given by \eqref{eq:normal_form_hom}. This results in the claimed linear Asymptotics \eqref{normal_flow_calculated}, when we use that the non-linear part vanishes: $P\big(u,\phi(u) \big) = P(u,0,0) = 0$.
\end{proof}
\end{theorem}

\section{Numerical evaluation of the spectrum of $\boldsymbol{\mathcal{L}}$}
\label{ch:spectrum_evans}

As announced in our discussion in Section \ref{ch:discussion}, we analyze the spectral stabilty of the critical traveling wave. The theoretical background is presented in \cite{SANDSTEDE_stability_traveling, Ghazaryan_overview}, the details about the computational approach are presented by Barker \textit{et al.} \cite{Barker_Evans_computations}, we only describe the application in the present setting.

Here and from now on, $c=2$ and we denote as $a(z), i(z)$ the critical traveling wave with speed $c=2$ and $\ipinf = 0$. We denote the exponent of the weight-function as $\alpha>0$ and analyze the spectral stability of the critical traveling wave in the weighted $L^2$-Sobolev space $H^1_\alpha$, with norm $||f||_{H^1_\alpha} = ||f \cdot e^{\alpha z}||_{H^1}$.

We linearize the PDE around $a(z),i(z)$ and analyze the non-negative spectrum of the resulting linear operator $\mathcal{L}$, as defined in \eqref{Eq:Linearized_Operator}. For System \eqref{EQUA}, this operator $\mathcal{L}:H^2_\alpha(\mathds{R}) \times H^1_\alpha(\mathds{R}) \rightarrow H^2_\alpha(\mathds{R}) \times H^1_\alpha(\mathds{R})$ acts on a pair of functions $u \in H^2_\alpha, v \in H^1_\alpha$, which correspond to perturbations of $a$ and $i$, respectively:
\begin{align}
\begin{aligned} \label{Eq:Linear_Eq}
u &\mapsto  u'' + cu' + u(1 - (2a + i ))- va, \\
v &\mapsto  cv' + u(2a + i + r) + va.
\end{aligned}
\end{align}
As will see later, we only need to consider the point-spectrum of $\mathcal{L}$. Thus, for $\gamma \in \mathds{C}$ with $\mathfrak{Re}(\gamma) \geq 0$, we look for a function $U \in H^1_\alpha$ that solves $\mathcal{L} U = \gamma \cdot U$.

The operator $\mathcal{L}$ is equivalent to a first-order operator $\tilde{\mathcal{L}}:H^1_\alpha(\mathds{R}^3) \rightarrow H^1_\alpha(\mathds{R}^3)$, when we introduce an auxiliary variable for $u'$. We will omit the tilde. Now, $\gamma \in \mathds{C}$ lies in the point spectrum of $\mathcal{L}$ if and only if there exists a function $U: \mathds{R} \rightarrow \mathds{C}^3, U \in H^1_\alpha$, which solves
\begin{align}
\frac{d}{dz} U = M(z, \gamma) \cdot U,
&&M(z, \gamma) := \begin{pmatrix}
0 & 1 & 0 \\
\gamma + 2a(z) + i(z) -1 & -c & a(z) \\
- \frac{2a(z)+i(z)+r}{c} & 0 & \frac{\gamma -a(z)}{c}
\end{pmatrix}.
\label{Eq:linear_problem_ODE}
\end{align}
It can easily be seen that the matrix $M(+\infty, \gamma)$ has eigenvalues
\begin{align}
\beta_1 &= \frac{\gamma}{c}, \hspace{0.5cm}  && \beta_{2} = - \frac{c}{2} +\sqrt{\gamma}, \hspace{0.5cm} && \beta_{3} = - \frac{c}{2} - \sqrt{\gamma}, \label{Eq:Eigenvalues_Lambda+}
\intertext{and the matrix $M(-\infty, \gamma)$ has eigenvalues}
\beta_1 &= \frac{\gamma}{c}, \hspace{0.5cm}  && \beta_{2} = - \frac{c}{2} +\sqrt{2+\gamma}, \hspace{0.5cm} && \beta_{3} = - \frac{c}{2} - \sqrt{2+\gamma}.  \label{Eq:Eigenvalues_Lambda-}
\end{align}

If $U \in H^1_\alpha$, then $W(z) := U(z) \cdot e^{\alpha z}$ is bounded and vanishes. The function $W(z)$ fulfills
\begin{align}
W'(z) = \big( M(z, \gamma) + \alpha \cdot \mathds{1} \big) \cdot W(z). \label{Eq:H_alpha_funtion}
\end{align}
Remark that the matrix $M + \alpha \mathds{1}$ has the same eigenvectors as $M$, and that its eigenvalues are shifted by $\alpha$ when compared to $M$. If $M(\pm \infty) + \alpha \mathds{1}$ has no eigenvalues with zero real-part, the theory of \textit{exponential dichotomies} implies that any bounded solution $W(z)$ must vanish exponentially fast as $z \rightarrow \pm \infty$, and that it asymptotically approaches the unstable (resp. stable) manifold of the constant matrix $M(- \infty, \gamma)$ as $z \rightarrow - \infty$ (resp. $M(+\infty,\gamma)$ as $z \rightarrow + \infty$) \cite{SANDSTEDE_stability_traveling}. Therefore, a bounded solution $W$ exists if and only if the trajectories that emerge from these manifolds intersect. This allows us to compute the Evans-function: it is a determinant that evaluates to zero if and only if the solutions that decay at $-\infty$ and those that decay at $+\infty$ are somehow linearly dependent, and thus intersect.

We investigate the case $\alpha = \frac{c}{2}$, which is equal to the rate of convergence of the wave as $z \rightarrow + \infty$, up to a sub-exponential term, see Theorem \ref{Main_Theorem}. For $\alpha = \frac{c}{2}$, then within the region $\{\mathfrak{Re}(\gamma) \geq 0, \gamma \neq 0\}$ the following holds: the dimension of the unstable space of $M(-\infty,\gamma)+ \frac{c}{2} \cdot \mathds{1}$ is given by $k_-=2$, and the dimension of the stable space of $M(+\infty,\gamma)+ \frac{c}{2} \cdot \mathds{1}$ is given by $k_+=1$. This can easily be deduced from the corresponding Eigenvalues \eqref{Eq:Eigenvalues_Lambda+}, \eqref{Eq:Eigenvalues_Lambda-}, which do not cross the imaginary axis. The values $k_-$ and $k_+$ add up to the dimension of the ODE \eqref{Eq:linear_problem_ODE}. We say that $\{\mathfrak{Re}(\gamma) \geq 0, \gamma \neq 0\}$ is contained in the \textit{region of consistent splitting}. This implies that the non-negative part of the spectrum of $\mathcal{L}$ is contained in the point spectrum of the operator, which is a standard result \cite{SANDSTEDE_stability_traveling, SATTINGER_stability}. Within the region of consistent splitting, we can define the Evans-function $E(\gamma)$.

Given $\gamma$ with $\mathfrak{Re}(\gamma) \geq 0, \gamma \neq 0$, we let $X(z)$ be the unique solution of Eq. \eqref{Eq:H_alpha_funtion} that vanishes at $z = +\infty$, and let $Y_1(z),Y_2(z)$ span the two-dimensional space of solutions of Eq. \eqref{Eq:H_alpha_funtion} that vanish at $z = -\infty$. The Evans-function is defined as
\begin{align}
E( \gamma ) := \text{det} \big ( Y_1(z) \big | Y_2(z) \big | X(z) \big )  \Big |_{z=0}. \label{Eq:Evans}
\end{align}

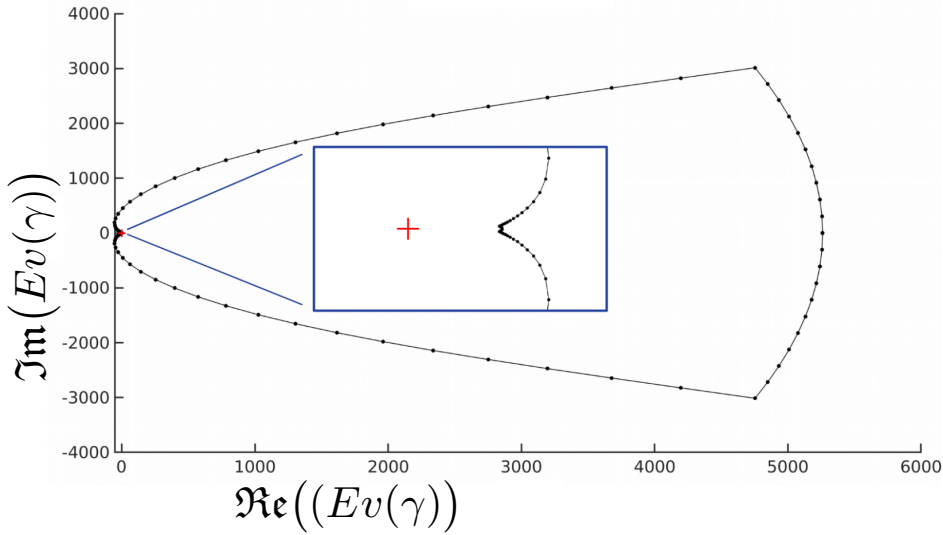
\begin{figure}[h]
\vspace{2.6cm}
 	\centering
 	\begin{minipage}[c]{0.95\textwidth}
\begin{picture}(100,100)
	\put(10,0){\includegraphics[width=\textwidth]{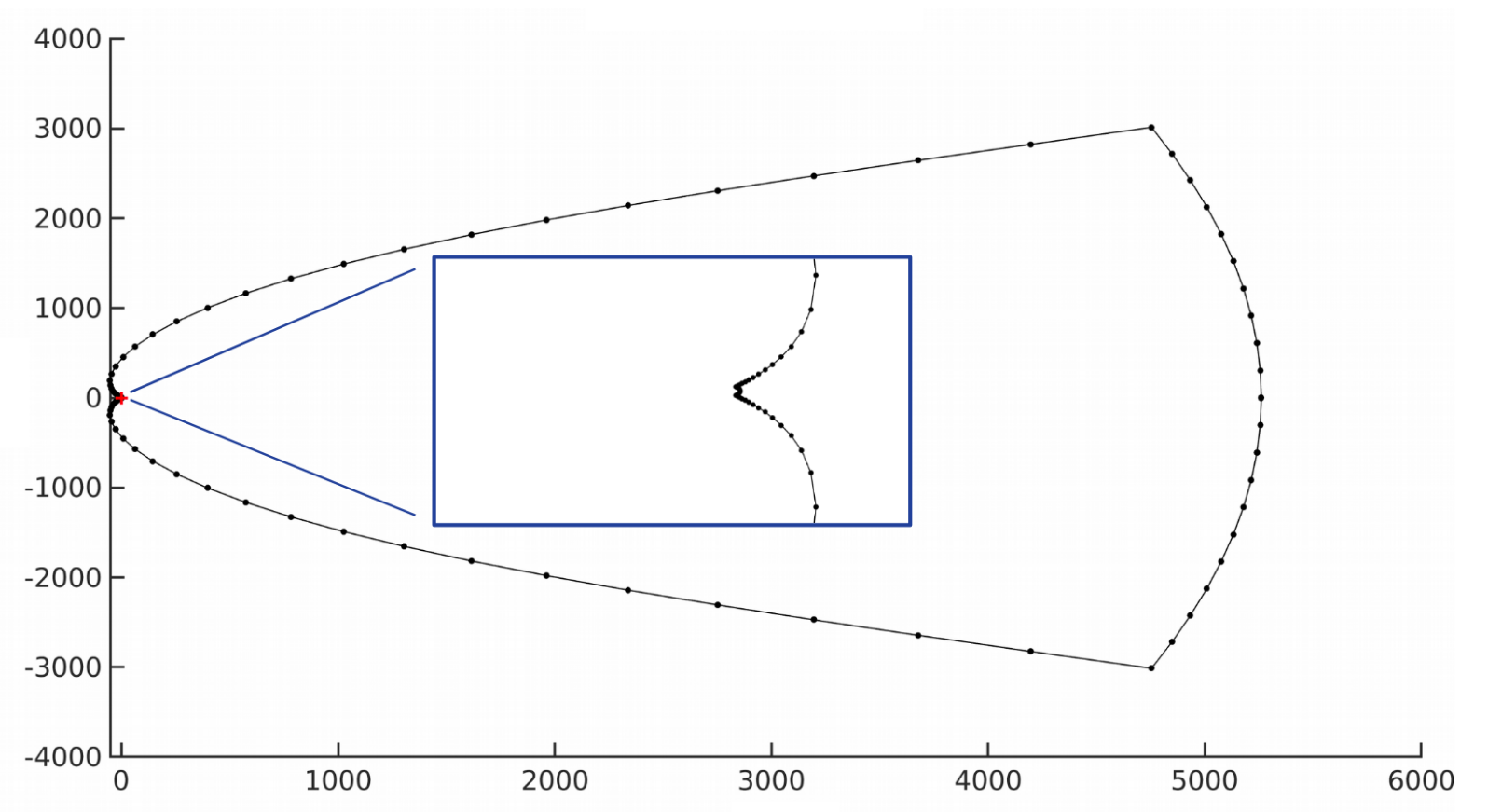}}
	\put(80,-12){\scalebox{1.5}{$\mathfrak{Re}\big((Ev(\gamma)\big)$}}
	\put(-5,30){\rotatebox{90}{ \scalebox{1.5}{ $\mathfrak{Im}\big(Ev(\gamma)\big)$ }}}
	\put(140,94){\textcolor{Red}{$\boldsymbol{+}$}}
\end{picture}
	\end{minipage}
	\vspace{0.5cm}
\caption{Numerical evaluation of the Evans-function \eqref{Eq:Evans} $Ev(\gamma)$ for $r=0$ and $\gamma$ on the boundary of the Domain $S$, defined in \eqref{Eq:Eiganvalue_Region}. The Evans-function for $r=1$ is very similar. The origin is marked with a small red cross. The graph does not enclose the origin and it can visually be seen that its winding number is equal to zero. We conclude that the Region $S$ contains no zeros of $Ev(\gamma)$.}
\label{Fig:Evans}
\end{figure}

It holds that $E(\gamma) = 0$ if and only if $\gamma$ lies in the point spectrum of $\mathcal{L}$. Moreover, $E(\gamma)$ is analytic if $X,Y_1,Y_2$ are chosen such that they are analytic in $\gamma$ \cite{SANDSTEDE_stability_traveling}. Thus it suffices to calculate $E(\gamma)$ along the boundary of a domain: the winding number along this contour then corresponds to the number of zeros inside the domain. We use this to verify that there are no zeros of $E(\gamma)$ within the set
\begin{align}
S:= \Big \{ \gamma \in \mathds{C} \, \Big | \, \mathfrak{Re}(\gamma) \geq 0, \, 10^{-3} \leq |\gamma| \leq 1000
\Big \}, \label{Eq:Eiganvalue_Region}
\end{align}
where we keep a small distance from the origin for numerical reasons and hope that there are no unexpectedly large eigenvalues. We want to remark that for partially diffusive systems, no general a priori upper bound for the size of the eigenvalues with non-negative real-part has been found yet, which would allow for a numerical proof of spectral stability. It may be possible to generalize the approach in \cite{Lattanzio_Dispersive_Shocks}. Simple energy estimates exist for traveling waves of diffusive systems, see e.g. chapter 6 in \cite{OzbagSchecter2018}.

The various numerical challenges that arise when computing the Evans-function, as well as their solutions, are described in detail by Barker \textit{et al.} \cite{Barker_Evans_computations}, who also suggest using their library STABLAB \cite{Barker_Stablab}. We gratefully followed this suggestion, and computed the left-adjoint Evans-function, a slight modification which is numerically advantageous in the present setting \cite{Barker_Evans_computations}. The result is presented in Figure \ref{Fig:Evans} and yields a strong evidence that the critical wave is spectrally stable in $H^1_{c/2}$.

\section{Rescaling the general system} \label{ch:Rescaling}
Let $r_S, r_A, D > 0$ and $r_I \geq 0$, and consider the reaction-diffusion system
\begin{equation}
\begin{aligned} \label{Eq:General_PDE}
A_t &= D \cdot A_{xx} + r_A A - r_S A(A+I), \\
I_t &= r_I A + r_S A(A+I),
\end{aligned}
\end{equation}
which is the general form of System \eqref{EQUA}. There exists a linear one-to-one correspondence to the normalized form. Therefore, we rescale time and space, $s := r_A \cdot t, \, y := \sqrt{D / r_A} \cdot x$, and also the densities of the particles, $\bar{A} := A \cdot r_S/r_A, \, \bar{I} := I \cdot r_S/r_A $. The rescaled dynamics of System \eqref{Eq:General_PDE} follow
\begin{equation}
\begin{aligned}
\bar{A}_s &= \bar{A}_{yy} + \bar{A} - \bar{A} ( \bar{A} + \bar{I} ) ,  \\
\bar{I}_s &= \frac{r_I}{r_A} \bar{I} + \bar{A} (\bar{A} + \bar{I} ),
\end{aligned}
\end{equation}
which is equivalent to System \eqref{EQUA} with $r = \frac{r_I}{r_A}$. In view of this, we can easily formulate a parameter-dependent version of Theorem \ref{Main_Theorem}:

\begin{theorem} \label{Main_Theorem_parameters}
Let $r_S, r_A, D > 0$ and $r_I \geq 0$, and consider the System \eqref{Eq:General_PDE} and a wave-speed $c>0$. Set
\begin{align}
\oc : = \max \big \{ 0, \frac{1}{r_S} \big ( r_A-\frac{c^2}{4D} \big ) \big \}.
\end{align}
For each pair $\iminf, \ipinf \in \mathds{R}^+$ such that
\begin{align}
\ipinf \in [\oc,\frac{r_A}{r_S}), \qquad \iminf = \frac{2 \cdot r_A}{r_S} - \ipinf, 
\end{align}
there exists a unique bounded and positive traveling wave $a,i$ with speed $c$ such that
\begin{align}
\lim_{\ta \rightarrow \pm \infty} & a(z) = 0, \qquad  \lim_{\ta \rightarrow \pm \infty} i(z) = i_{\pm \infty}.
\end{align}
If $ \frac{c^2}{4D} + r_S \cdot \ipinf - r_A = 0$, then convergence as $z \rightarrow + \infty$ is sub-exponentially fast and of order $ z \cdot e^{ - \frac{c}{2D} z}$. If $ \frac{c^2}{4D} + r_S \cdot \ipinf - r_A  > 0$, then convergence as $z \rightarrow + \infty$ is exponentially fast. Convergence as $z \rightarrow - \infty$ is exponentially fast in all cases. The corresponding rates are
\begin{align}
\mu_{\pm \infty} = - \frac{c}{2D} + \sqrt{ \frac{c^2}{4D^2} + \frac{r_S \cdot i_{\pm \infty} -r_A}{D} }.
\end{align}
Moreover, these are all bounded, non-negative, non-constant and twice differentiable solutions of Eq. \eqref{WAVE_EQ}.
\end{theorem}
In particular, for an invading front where $i \rightarrow 0$ as $z \rightarrow + \infty$, the remaining density of particles at the back of the wave is given by $\iminf = 2 \cdot \frac{r_A}{r_I}$.

\end{document}